\documentclass[a4paper,10pt,leqno, english]{smfart}
\usepackage{aeguill}
\usepackage{enumerate}
\usepackage{enumitem}
\usepackage{amssymb,amsmath,latexsym,amsthm}
\usepackage[T1]{fontenc}
\usepackage{smfthm}      
\usepackage{geometry}   
\usepackage{url}
\usepackage[frenchb, english]{babel}
\usepackage[utf8]{inputenc}   
\usepackage{mathrsfs}
\usepackage{xcolor}
\usepackage{comment}
\definecolor{violet}{rgb}{0.0,0.2,0.7}
\definecolor{rouge2}{rgb}{0.8,0.0,0.2}
\usepackage{etoolbox}
\usepackage[normalem]{ulem}
\patchcmd{\thebibliography}{*}{}{}{}
\pretocmd\thebibliography{\csname c@secnumdepth\endcsname=-2 }{}{}
\usepackage{hyperref}

\usepackage{relsize}
\usepackage[arrow,curve,matrix]{xy}
\usepackage{tikz-cd}

\setdescription{labelindent=\parindent, leftmargin=2\parindent}
\setitemize[1]{labelindent=\parindent, leftmargin=2\parindent}
\setenumerate[1]{labelindent=0cm, leftmargin=*, widest=iiii}

\numberwithin{figure}{section}

\hypersetup{
    bookmarks=true,         
    unicode=false,          
    pdftoolbar=true,        
    pdfmenubar=true,        
    pdffitwindow=false,     
    pdfstartview={FitH},    
    pdftitle={},    
    pdfauthor={},     
    colorlinks=true,       
   linkcolor=rouge2,          
    citecolor=violet,        
    filecolor=black,      
    urlcolor=cyan}           

\numberwithin{equation}{section}

\DeclareMathOperator{\codim}{codim}

\DeclareMathOperator{\Gal}{Gal}
\DeclareMathOperator{\rank}{rank}
\DeclareMathOperator{\reg}{reg}
\DeclareMathOperator{\sing}{sing}
\DeclareMathOperator{\supp}{Supp}
\DeclareMathOperator{\A}{A}
\DeclareMathOperator{\orb}{orb}
\DeclareMathOperator{\Neron}{N}
\DeclareMathOperator{\red}{red}

\newcommand{\CC}{\mathbb{C}}
\newcommand{\Q}{\mathbb{Q}}

\newcommand{\N}{\mathbb{N}}

\renewcommand{\d}{\partial}

\newcommand{\vp}{\varphi}

\renewcommand{\O}{\mathcal{O}}

\newcommand{\ep}{\varepsilon}
\renewcommand{\epsilon}{\varepsilon}

\newcommand{\la}{\langle}

\newcommand{\ra}{\rangle}

\renewcommand{\ge}{\geqslant}
\renewcommand{\le}{\leqslant}
\renewcommand{\leq}{\leqslant}
\renewcommand{\geq}{\geqslant}
\newcommand{\Ric}{\mathrm{Ric} \,}
\newcommand{\om}{\omega}

\newcommand{\ddc}{dd^c}

\newcommand{\ce}{\chi_{\eta}}

\newcommand{\sN}{\mathscr{N}}
\newcommand{\E}{\mathscr{E}}
\newcommand{\F}{\mathscr{F}}
\newcommand{\sH}{\mathscr{H}}
\newcommand{\sL}{\mathscr{L}}
\newcommand{\sW}{\mathscr{W}}

\newcommand{\Supp}{\mathrm {Supp}}
\newcommand{\tr}{\mathrm{tr}}

\newcommand{\sT}{\mathscr{T}}

\renewcommand{\F}{\mathscr{F}}
\newcommand{\G}{\mathscr{G}}
\newcommand{\sA}{\mathscr{A}}

\newcommand{\ud}{|u|^2+\lambda^2}

\newcommand{\omte}{\om_{t,\ep}}

\newcommand{\omx}{\om_{\wt Y}}
\newcommand{\wt}{\widetilde}
\newcommand{\omxt}{\om_{\wt X}}

\theoremstyle{plain}
\newtheorem{theorem}{Theorem}[section]

\newtheorem{propo}[theorem]{Proposition}

\newtheorem{lemma}[theorem]{Lemma}

\newtheorem{defin}[theorem]{Definition}
\theoremstyle{definition}
\newtheorem{example}[theorem]{Example}
\newtheorem{remark}[theorem]{Remark}
\newtheorem{notation}[theorem]{Notation}

\newtheorem{claim}[theorem]{Claim}
\newtheorem{set-up}[theorem]{Set-up}
\newtheorem{subclaim}[theorem]{Subclaim}
\newtheorem{assumption}{Assumption}
\newtheorem*{assumptionS}{Assumption}

\setlist[enumerate]{label=(\thetheorem.\arabic*), before={\setcounter{enumi}{\value{equation}}}, after={\setcounter{equation}{\value{enumi}}}}

\theoremstyle{plain}
\newtheorem{bigthm}{Theorem}

\renewcommand{\theequation}{\thesection  .\!\! \arabic{equation}}

\begin{document}

\begin{abstract}
After establishing suitable notions of stability and Chern classes 
for singular pairs, we use K\"ahler-Einstein metrics with conical and 
cuspidal singularities to prove the slope semistability of orbifold 
tangent sheaves of minimal log canonical pairs of log general type. 
We then proceed to prove the Miyaoka-Yau inequality for all 
minimal pairs with standard coefficients. Our result in particular provides 
an alternative proof of the Abundance theorem for threefolds 
that is independent 
of positivity results for cotangent sheaves established by Miyaoka.
\end{abstract}

\title[Stability and Miyaoka-Yau inequality]{Orbifold Stability and Miyaoka-Yau Inequality for minimal pairs}

\date{\today}
\author{Henri Guenancia}
\address{Institut de Mathématiques de Toulouse; UMR 5219, Université de Toulouse; CNRS, UPS, 118 route de Narbonne, F-31062 Toulouse Cedex 9, France}
\email{henri.guenancia@math.cnrs.fr }
\urladdr{https://hguenancia.perso.math.cnrs.fr/}
\author{Behrouz Taji}
\address{School of Mathematics and Statistics F07, The University of Sydney, NSW 2006, Australia}
\email{behrouz.taji@sydney.edu.au} 
\urladdr{http://www.maths.usyd.edu.au/u/behrouzt/}

\thanks{Henri Guenancia is partially supported by NSF Grant DMS-1510214. Behrouz Taji was partially supported by the DFG-Graduiertenkolleg GK1821 “Cohomological Methods in Geometry” at Freiburg.}

\maketitle

\tableofcontents

\section{Introduction}
\label{section1-Intro}

In 1954 Calabi conjectured that a compact complex manifold with negative first Chern class
$c_1(X)<0$ admits a K\"ahler-Einstein metric. This conjecture was famously settled by 
Aubin \cite{Aubin} and Yau \cite{Yau78} leading to many remarkable applications in algebraic geometry for those manifolds. An 
important consequence was the celebrated Miyaoka-Yau inequality (cf. \cite{Yau77}): 

\begin{equation}\label{eq:Yau1}
\bigl( 2(n+1)\cdot c_2(X) - n\cdot c_1^2(X)\bigr) \cdot (-c_1(X))^{n-2} \geq 0,  
\end{equation} 
where $n=\dim (X)$. 

The first main result of the current paper is the generalization of Inequality~\eqref{eq:Yau1} 
to the case of all minimal models. 

 \begin{bigthm}
  \label{thm:1stMY}
 \emph{ \textbf{(The Miyaoka-Yau inequality for minimal models)} }
  
  \noindent
Any minimal model of dimension $n$ verifies the Miyaoka-Yau inequality: 
   \begin{equation}\label{MY-ineq}
    \bigl( 2(n+1)\cdot c_2(X) - n\cdot c_1^2(X) \bigr) \cdot (-c_1(X))^{n-2}\geq 0, 
   \end{equation}
 \end{bigthm}
 
 A minimal model $X$ is a normal complex projective variety with only terminal singularities 
whose canonical divisor is a $\Q$-Cartier nef divisor. 
According to standard results and conjectures in the Minimal Model Program minimal 
varieties exist, at least conjecturally, in the birational class of any non-uniruled projective
manifold. In this context, thanks to his celebrated result on the so-called generic 
semi-positivity of cotangent sheaves, Miyaoka proved that the inequality 

\begin{equation}\label{ineq:MiyOriginal}
\bigl(  3c_2(X)  - c_1^2(X) \bigr) \cdot H^{n-2}  \geq 0,
\end{equation}
holds, for any ample divisors $H\subset X$, cf.~\cite{Miyaoka}. In this light, the inequality~\eqref{MY-ineq}
can be seen as bridging the gap between the two inequalities of Miyaoka (\ref{ineq:MiyOriginal}) 
and Yau (\ref{eq:Yau1}),
when the polarization is chosen to be the canonical one. 

More generally, minimal models are studied in the setting of pairs $(X,D)$
where a projective variety $X$ is considered together with a divisor $D$
such that $K_X+D$ is nef and $(X,D)$ has only ``mild" singularities. 
Naturally one would like to generalize the inequality~\eqref{eq:Yau1} in this setting. 

Generalization of Miyaoka-Yau inequalities have attracted a lot of attention over the last thirty years, with major contributions due to Tsuji, R. Kobayashi, Tian-Yau, Simpson, Megyesi, Y. Zhang, Song-Wang, Greb, Kebekus, Peternell together with the second author, 
to cite only a few. We have tried to render a brief account of these contributions in the last section of the introduction. 

One of the remaining cases of interest for this inequality is that of log canonical pairs $(X,D)$ where $K_X+D$ is nef. 
Here the situation gets significantly more complicated; even when defining the correct notion for
Chern classes. In fact in the most general setting, where all possible rational coefficients for 
$D$ are allowed, 
it is not even clear how one should define higher Chern classes. 
However, under the assumption of Theorem~\ref{thm:main-ineq} below (which seem to be somehow the maximally singular cases where orbifold Chern classes can be defined, cf the remarks under the aforementioned theorem), there exists a Zariski open subset 
$X^\circ$ with $\codim_{X}(X\backslash X^\circ)\geq 3$ 
where it is possible to find a collection of local smooth charts encoding the structure of 
the boundary $D$. Then, one can follow constructions similar to those of Mumford~\cite{MR717614} for 
$\mathcal Q$-varieties to define $c_2(X,D)$ cycle theoretically. 
With this definition at hand, we prove the Miyaoka-Yau inequality for minimal dlt pairs whose boundary has standard coefficients:
\newpage

  \begin{bigthm}
  \label{thm:main-ineq}
\emph{  \textbf{(The Miyaoka-Yau inequality for minimal dlt pairs)}}
  
  \noindent
  Let $X$ be a normal, projective variety of dimension $n$ and let $D$ be an effective $\Q$-divisor with standard coefficients, i.e. $D = \sum (1 -\frac{1}{n_i})\cdot D_i$ with $n_i \in \mathbb N \cup \{\infty\}$. Assume that 
  \begin{enumerate}
  \item[$(i)$] Each component $D_i$ of $D$ is a $\Q$-Cartier divisor;
  \item[$(ii)$] The pair $(X,D)$ has dlt singularities;
  \item[$(iii)$] The $\Q$-line bundle $K_X+D$ is nef.
  \end{enumerate}   
Let $\nu$ denote the \emph{numerical Kodaira dimension} of $K_X+D$. Then, for any ample divisor $H$ in $X$, the inequality
  \begin{equation}\label{eq:MYIneqMain}
    \bigl( 2(n+1)\cdot c_2(X,D) - n\cdot c_1^2(X,D) \bigr) \cdot (K_X+D)^{i} \cdot H^{j}\geq 0.
    \end{equation}
  holds, where $i=\min(\nu,n-2)$ and $j=n-i-2$. 
  \end{bigthm}

 We follow the usual convention by setting ``$n_i =\infty$" when the expression $(1-\frac{1}{n_i})$ 
 is equal to $1$.
 Also recall that the numerical Kodaira dimension $\nu(B)$ of a nef $\Q$-divisor $B$ is defined 
 by $\nu(B):= \max \{ m\in \N \; \big| \; c_1(B)^m\neq 0    \} $. 
 For the definition of the various types of singularities that appear in this paper we refer to~\cite[\S 2.3]{KM}.
 We recall that klt singularities are dlt. 
 
 \subsection*{A few remarks about Theorem~\ref{thm:main-ineq}} We would like to detail a few points in Theorem~\ref{thm:main-ineq}.

$(a)$ First, the restriction on the singularities is essential to guarantee the existence of suitable covers, and 
consequently a good notion
for $c_2(X,D)$; we refer to Section~\ref{examples} for an in-depth discussion.   

$(b)$ We can extend Theorem~\ref{thm:main-ineq} to the following case:
\begin{enumerate}
\item[$\ast$] The pair $(X,D)$ is log smooth, where $D= \sum d_i\cdot D_i$,   $d_i\in [0,1]\cap \Q$.
\end{enumerate}

$(c)$ The strategy of the proof of Theorem~\ref{thm:main-ineq} relies, in a crucial way, on an approximation process which amounts to replacing $(X,D)$ by $(X,D+\frac 1m H)$ for some suitable ample divisor $H$, and then passing to the limit when $m \to \infty$. This approximation, detailed in Proposition~\ref{prop:OrbiLc} is required for the following two things.
 
\begin{enumerate}
\item[$\ast$] To be able to deal with the case where $K_X+D$ is nef but not big.

\item[$\ast\ast$] To clear the poles of the Higgs field in the last step of the proof: as the new pair is klt, the canonical Higgs field on the cover has no poles along the divisor (rather zeros), allowing us to apply Simpson's results on the Kobayashi-Hitchin correspondence. \\
\end{enumerate}

 The remark above highlights the importance of having a good (intrinsic) theory of orbifold sheaves (and related objects) in the case of boundary with \emph{arbitrary rational coefficients}, which forms the main bulk of the preliminary sections~\ref{section2-Prelim} 
and~\ref{section3-compatibility} of the current paper.  
These constructions provide sufficient flexibility that is crucial in dealing with the approximation process 
mentioned above. Indeed it is in this context that we establish 
slope semistability for the tangent sheaf of minimal models; a result that turns out to be an essential tool in proving Theorem~\ref{thm:main-ineq}.

\

\begin{bigthm}
 \label{thm:SS}
 \emph{\textbf{(Semistability of the orbifold tangent sheaf of minimal lc pairs).}}
 
 \noindent
 The orbifold tangent sheaf of any minimal log canonical pair of log general type $(X,D)$ is slope semistable 
 with respect to $K_X+D$.
 \end{bigthm}

The general strategy to prove Theorem~\ref{thm:SS} is inspired from \cite{CP2} and \cite{GSS}, and the main analytical input is the theory of \textit{conical/cuspidal} metrics, cf \S \ref{mixed}. These metrics are the logarithmic (or pair) analogue of Kähler metrics and they provide canonical \--though possibly singular\-- hermitian metrics on the orbifold tangent sheaf of a pair $(X,D)$, say when $K_X+D$ is ample. These metrics are the key to derive geometric properties of the orbifold tangent sheaf (like semistability) knowing only positivity/negativity properties of its determinant.
However, when the pair $(X,D)$ is singular, these metrics are unfortunately too singular to carry over the existing analysis in the (log) smooth case. So one needs to regularize the metrics on a resolution and control the resulting error terms. 

The control of the error term is given by Lemma~\ref{int:zero}; from a technical point of view, this constitutes the core of the proof. This lemma is the equivalent in this more general context of \cite[Lemma 3.7]{GSS}. However, the strategy of its proof is completely different and somewhat simpler. It relies on pluripotential theory and more specifically on the notion of strong convergence developed e.g. in \cite{BBEGZ}. We explain in Remark~\ref{remark} why a new strategy was actually needed. 

The last step of the proof of Theorem~\ref{thm:SS} is to relate the semistability of the orbifold tangent sheaf of a resolution to the one of $X$. This is a place where it is crucial to have defined the orbifold tangent sheaf in an intrinsic way (that is, not with any particular choice 
for a cover but rather the possibility to work with all adapted covers at the same time). \\

 \subsection*{Application to the Abundance Conjecture}
 The inequality~\eqref{eq:MYIneqMain} in dimension $2$ and in the smooth setting 
 was established by Miyaoka through purely 
 algebraic methods.
 In higher dimensions a weaker inequality was famously proved, again by 
 Miyaoka (\cite{Miyaoka}), via his work on generic semipositivity of the cotangent sheaves of minimal 
 models, an approach that heavily depends on sophisticated characteristic-$p$ arguments. 
 Miyaoka's result in dimension $3$, and its generalization by Megyesi (\cite[Chapt.~10]{Kollar92}), 
 namely the inequality 
 $$
 c_2(X,D) \cdot (K_X+D) \geq 0
 $$
 for a minimal lc pair $(X,D)$ with klt $X$, were fundamental to the proof of Abundance conjecture for threefold 
 cf.~\cite{Kaw91} and~\cite{Kollar92}. In this light, Theorem~\ref{thm:main-ineq} provides an alternative 
 way for proving the Abundance Conjecture (in dimension $3$), 
 that is independent of generic positivity results for cotangent sheaves of minimal models.

 \subsection*{Structure of the paper.}
$^{ }$

$\bullet$ Sections~\ref{section2-Prelim} 
and~\ref{section3-compatibility} provide the suitable algebraic framework to work with sheaves on pairs $(X,D)$, where $D$ has rational coefficients. Roughly speaking, every lc pair $(X,D)$, with $X$ being klt, has a natural structure of a smooth Deligne-Mumford stack in codimension two
(see Subsection~\ref{subsect:OrbiLc}). Such structures can then be 
endowed with linearized sheaves; the \emph{orbifold sheaves}. In particular, and inspired by the works of Campana, 
we can naturally define an orbifold tangent sheaf for the pair $(X,D)$ (Definition~\ref{def:adaptedsheaf}). 
The Chern classes of such orbifold sheaves can then be defined in an orbifold sense 
(see Section~\ref{subsect:OrbiChern}).

$\bullet$ In Section~\ref{section4-semistability}, after recalling the basic definitions about conical/cuspidal metrics, we then use 
the regularity results of \cite{GP} about conical/cuspidal Monge-Ampère equations 
to derive the semistability of the orbifold tangent sheaf of any minimal lc pair, in the spirit of \cite{GSS}. Even though the global approach is similar, one of the key estimates (Lemma \ref{int:zero}) requires a new input, cf Remark \ref{remark}.

$\bullet$ In Section~\ref{section5-MYInequality}, and by following a similar strategy to that of~\cite{GKPT15},
we prove Theorem~\ref{thm:main-ineq} using Theorem~\ref{thm:SS} to construct a 
a stable orbi-Higgs sheaf whose Bogomolov-Gieseker Chern class discriminant 
is equal to that of Miyaoka-Yau for the orbifold tangent sheaf of $(X,D)$.
At this point Simpson's result on the Kobayashi-Hitchin correspondence 
for Higgs bundles can be used to prove the Miyaoka-Yau inequality for $(X,D)$. \\

 \subsection*{Previously known results}
 As we explained above already, the Miyaoka-Yau inequality and its various generalizations have been 
 intensely studied. There are been different types of generalization so far: 

$\cdotp$ By relaxing the assumption on ampleness of $K_X$ and replacing it with $K_X$ nef and big. The first approach seems to be due to Tsuji \cite{Tsuji88} using orbifold metrics; later, Y. Zhang \cite{MR2497488} gave a proof using the Kähler-Ricci flow relying  on the scalar curvature bound of Z. Zhang \cite{ZhangZhou}. Finally, Song and Wang \cite{SW} used the regularity results of \cite{JMR} about conical metrics to reprove that inequality. The idea that conical metrics could be used to generalize Miyaoka-Yau inequality has been suggested by Tian \cite{Tian94} already twenty years ago. 

$\cdotp$ In the setting of log smooth pairs $(X,D)$ with standard coefficients, the Miyaoka-Yau inequality has been obtained by R. Kobayashi \cite{KobR} (assuming $D$ reduced) and Tian-Yau \cite{TY87} (in general). The proofs rely on the generalization of Aubin-Yau theorem in this setting, where the suitable geometry involves orbifold and cuspidal metrics. Using conical Kähler-Einstein metrics, Song-Wang \cite{SW} (partially) generalized these results to the case where $(X,D)$ is log smooth and $D$ is smooth with arbitrary real coefficients in $(0,1)$.

$\cdotp$ For singular klt surfaces, and more generally for log canonical pairs $(S,C)$ where $S$ is a klt surface and $C$ is a (reduced) curve such that $K_S+C$ is nef, the Miyaoka-Yau inequality was showed by Megyesi \cite[Chap. 10]{Kollar92}. 

$\cdotp$ In another direction, Simpson \cite{MR1179076} observed that Miyaoka-Yau inequality can be (almost formally) deduced from Bogomolov-Gieseker inequality for semistable bundles using the Higgs bundle $\Omega_X\oplus \mathcal O_X$. 

$\cdotp$ Using Simpson's much more robust approach and the first author's semistability result \cite{GSS}, Greb, Kebekus, Peternell and the second author have recently proved the Miyaoka-Yau inequality for projective minimal varieties of general type with klt singularities, cf. \cite{GKPT15}.

 \subsection*{Acknowledgements} 
 Both authors would like to thank Daniel Greb, Stefan Kebekus, Robert Lazarsfeld, Mihai P\u{a}un, Thomas Peternell
 and Jason Starr for helpful discussions. The authors also owe a debt of gratitude to the 
 anonymous referees for numerous helpful comments and for pointing out some errors in an earlier version 
 of this paper.

 \section{Preliminaries on orbifold sheaves, Chern classes and stability}
\label{section2-Prelim}

We begin the current section by reviewing some basic orbifold constructions.
These constructions will then be used in Subsection~\ref{subsect:OrbiChern} to introduce relevant notions 
of stability and Chern classes that are crucial for the rest of the paper. 

 \begin{defin}[Pairs]\label{def:pairs} 
 A pair $(X,D)$ consists of a normal quasi-projective variety $X$ 
 and a divisor $D=\sum d_i\cdot D_i$, where $d_i=(1-\frac{b_i}{a_i}) \in [0,1] \cap \Q$ with $a_i \in \N \cup\{+\infty\}$, $b_i\in \N$,  and each $D_i$ is prime. We say that $(X,D)$ has standard coefficients if $b_i=1$, for every $i$. In the following, $a_i$ and $b_i$ are always assumed to be relatively prime. 

 \end{defin}

 \begin{defin}[Pull-back of Weil divisors]\label{def:pull-back}
 Let $f:Y \to X$ be a finite and surjective morphism between normal 
 quasi-projective varieties $Y$ and $X$. For every Weil divisor $D\subset X$, we define the pull-back
 $f^*(D)$ by the Zariski-closure of $(f|_{f^{-1}(X_{\reg})})^*(D|_{X_{\reg}})$.
 \end{defin}

 \subsection{Adapted morphisms}
 We now recall a notion of morphism that encodes fractional structure in the boundary 
 divisor of a given pair. See for example~\cite[Sect.~2]{JK} or~\cite[Sect.~2.6]{CKT} for similar definitions
 and more examples.
  
 \begin{defin}[Adapted morphisms]\label{def:orbi-morph}
 Let $(X,D)$ be a pair as in Definition~\ref{def:pairs}. 
 A finite, Galois and surjective morphism $f:Y \to (X,D)$ is called 
 adapted to $(X,D)$ if 
 \begin{enumerate}
 \item The variety $Y$ is a normal and quasi-projective.
 \item  For every $D_i$, with $d_i \neq 1$, there exists $m_i\in \N$ and a reduced Weil divisor $D_i'\subset Y$ such that 
 $f^*(D_i)= (m_i a_i)\cdot D_i'$.
 \item The morphism $f$ is \'etale over the generic point of $\Supp(\lfloor D \rfloor )$.
 
 \end{enumerate}
 Furthermore, we say that $f$ is \emph{strictly adapted} if $m_i=1$, for all $i$.

 \end{defin}
 
 It will be important for subsequent constructions to work with adapted morphisms 
 that are, in codimension one, only ramified along a given boundary divisor. For 
 this we introduce the notion of \emph{orbi-\'etale morphisms}.
 
 \begin{defin}[Orbifold-\'etale morphisms]\label{def:OrbiEtale}
 Given a pair $(X,D)$,
 we call a strictly adapted morphism $f: Y\to (X,D)$ orbifold-\'etale, if the divisorial part of the branch locus of 
 $f$ is contained in $\Supp(D)$.
 
 \end{defin}

 \subsection{Orbifold structures}
 In this subsection we give a construction of local charts adapted to pairs, which we will use for 
 the orbifold notions of stability and Chern classes introduced later on in the 
 current section. 
 
 \begin{defin}[Local orbifold structures]\label{def:LocalStructure}
 Let $U_x\subseteq X$ be a Zariski open neighbourhood of $x\in X$ equipped with 
 a surjective, \'etale (quasi-finite) morphism $\sigma_x: \widetilde U_x \to U_x$ and 
 a morphism $g_x:  X_x \to \widetilde U_x$ adapted to $\sigma_x^*(D)$:
 
 $$
 \xymatrix{
 X_x \ar@/^7mm/[rrrr]^{f_x}    \ar[rr]^{g_x}_{\text{adapted}}  &&  \widetilde U_x   \ar[rr]^{\sigma_x}_
 {\text{\'etale}}
   && U_x.
  }
 $$
 We call
 the ordered triple $(U_x, f_x, X_x)$ an orbifold structure at $x$.
 Furthermore, if $(X_x, f_x^*D)$ is log-smooth, we say
 that the orbi-structure defined by $(U_x, f_x, X_x)$
 is \emph{smooth}. 

  \end{defin}
 
 Let us emphasize that we do not require $\sigma_x$ to be finite and thus no such assumption has
 been made for $f_x$ either. 
 As we will see later this has to do with the fact that algebraic 
 klt varieties have algebraic quotient singularities in codimension two only in the étale topology, cf.~e.g. the proof of Proposition~\ref{prop:LcReduced}.  However, we note that $g_x$ is assumed to be finite.  
 
 \begin{defin}[Strict and \'etale orbifold structures]\label{def:Strict}
 In Definition~\ref{def:LocalStructure}, if $f_x$ is strictly adapted or orbi-\'etale, 
 we say that the orbi-structure at $x$ is, respectively, strict or \'etale.
 
 \end{defin}

 \begin{defin}[Global structures]\label{def:GlobalOrbi}
 Let $\mathcal C= \{ (U_{\alpha} , f_{\alpha}, X_{\alpha} )  \}_{\alpha\in I}$, where $I$ is 
 an index set, be a collection of ordered triples describing local orbi-structures on $X$. 
 Let $\alpha, 
 \beta \in I$ and define $X_{\alpha\beta}$ to be the normalization of the fibre
 product $(X_{\alpha} \times_{X} X_{\beta} )$ with 
 the associated commutative diagram:

 $$
  \xymatrix{
    X_{\alpha\beta} \ar[rr]^{g_{\beta\alpha}}  \ar[d]_{g_{\alpha\beta}} && X_{\beta}  \ar[d]^{f_{\beta}} \\
      X_{\alpha} \ar[rr]^{f_{\alpha}}       && X,
  }
  $$
 where $g_{\alpha\beta} : X_{\alpha\beta} \to X_{\alpha}$ and $g_{\beta\alpha}: X_{\alpha\beta} \to X_{\beta}$ 
 are the projection maps. We say that $\mathcal C$ defines an orbi-structure on the pair $(X,D)$, if the following holds. 
 \begin{enumerate}
 \item $X=\bigcup_{\alpha\in I} U_{\alpha} $.
 \item \label{item:compatible} The two morphisms $g_{\beta\alpha}$ and $g_{\alpha\beta}$ are \'etale
 in codimension one.
 
 \end{enumerate}
  
 \end{defin}

 We note that the latter assumption is equivalent to the condition that, 
 for each $\alpha, \beta\in I$, the two morphisms 
 $f_{\alpha}$ and $f_{\beta}$ 
 have the same branch locus in codimension one with equal ramification indices.
 
If the structure $\mathcal C$ is smooth, i.e. if the local orbi-structures $(U_{\alpha} , f_{\alpha}, X_{\alpha} )$ are smooth for  for any $\alpha\in I$, then it follows from Nagata's purity of branch locus that the varieties $X_{\alpha\beta}$ are smooth and the morphisms $g_{\alpha\beta}$ are étale, for all indices $\alpha, \beta \in I$. 
 
 The reader may also like to compare the data $\{(U_{\alpha}, f_{\alpha}, X_{\alpha}) \}_{\alpha}$ to the 
 definition of an algebraic Deligne-Mumford 
 stack with local isotropic groups being trivial at
 the generic point of $X_{\alpha}$ and  
 cyclic of order $a_i$ over the fractional components of $D$.
 Here $X$ should be thought of as the coarse moduli space.

 \begin{defin}
 We say that a pair $(X,D)$ admits an orbi-structure in codimension $k$ for some integer $k\in \mathbb N$ if there exists a Zariski open subset $X^\circ \subset X$ such that
  \begin{enumerate}
 \item $\codim_X(X\setminus X^\circ) \ge k$. 
 \item $(X^\circ,D|_{X^\circ})$ admits an orbi-structure.
 \end{enumerate}
 \end{defin}
 
 \begin{assumptionS}[Orbi-structures are smooth]
\label{smoothOS}
From here onwards and until the end of the paper, any orbi-structure on any pair $(X,D)$ will be assumed to be \textup{smooth}, unless otherwise stated.  
\end{assumptionS}

 Note that a variety $X$ admitting a (smooth) orbi-structure has necessarily finite quotient singularities in the étale and consequently 
 analytic topology.

 \subsubsection{Examples}

 Every pair $(X,D)$ with $D$ reduced, trivially admits an orbifold structure via the 
 identity map.
  We now give a list of examples that are relevant to the rest of our discussions 
  in the current article.

 \begin{example}[Orbi-structures for snc pairs]
 \label{ex:OrbiSmooth}
 Every pair $(X,D)$ with $X$ smooth and $D$ having a simple normal crossing support admits
 various orbi-structures. Indeed, for every $x\in X$ there exists a Zariski open subset 
 $U_x\subset X$ that can be endowed with a natural \'etale orbi-structure as 
 follows. Let $U_x$ be a Zariski neighbourhood of $x$ where each irreducible component 
 $\{(D_i|_{U_x})\}_{i=1}^{k}$ of $D-\lfloor D \rfloor$ is principal, given by the zero locus of $f_i\in \O_{U_x}$.
 Let $\{t_i\}_{i=1}^k$ parametrize each copy of $\CC$ in the cartesian product $\CC^k \times U_x$. Then, 
 the subvariety $V_x \subset \CC^k \times U_x$ defined by the zero locus of $\{(t_i^{a_i}- f_i)\}_{i=1}^k$
 admits a projection onto $U_x$ that is orbi-\'etale with respect to $(X,D)|_{U_x}$. 
 The existence of the orbi-étale structure now follows from repeating this construction
 for each $x\in X$.
 
 Such pairs also admit a global, but certainly non-canonical orbifold structure. More precisely, thanks to Kawamata's 
 construction, cf. e.g. \cite[Prop.~4.1.12]{PAG1}, every snc 
 pair $(X,D)$ admits a strict, orbi-structure $f: Y \to X$, which fails to be orbi-\'etale along 
 a non-unique, very ample divisor.
 \end{example}

 \begin{remark}
 \label{codim1}
 As a consequence of Example~\ref{ex:OrbiSmooth} above, any pair $(X,D)$ admits various orbi-structures 
 in codimension one. 
 \end{remark}

 \begin{example}[The normal case]\label{ex:QFactorial}
 In the normal case constructions similar to the smooth example
 exist. For example, thanks to~\cite[Prop.~2.38]{CKT}, there exists a (global)
 strictly adapted morphism $f: Y\to (X,D)$. We note that in this construction $f: Y\to X$ is 
 constructed as the finite morphism in the Stein factorization of the map $\widetilde Y\to X$,
 where $\widetilde Y$ is the smooth quasi-projective variety in the Kawamata covering 
 $\widetilde f: \widetilde Y \to (\widetilde X, \widetilde D)$ of a log-resolution 
 $\pi:\widetilde X\to X$  of $(X,D)$:
 
 $$
 \xymatrix{
 \widetilde Y \ar[d]   \ar[rr]^{\widetilde f}  &&   \widetilde X \ar[d]^{\pi} \\
 Y  \ar[rr]^{f}  &&                       X.
 }
 $$
 As such the morphism is branched in codimension one over $\Supp(D)$ and a divisor $H$.
 Moreover, if one assumes that each component $D_i$ of $D$ is $\Q$-Cartier, then one can guarantee that $H$ belongs to a very ample linear system in $X$. 
To see this one can argue as follows. By assumption, there exist $m_i \in \mathbb N$ such that for any index $i$, $m_i\cdot D_i$ can be written as the difference
of two very ample divisors $A_i$ and $B_i$: $m_i\cdot D_i = A_i - B_i$.
Let $E$ an effective exceptional divisor such that $-E$ is ample
over $X$ and that, for sufficiently small $\epsilon \in \mathbb Q^+$, the two divisors 
 $(\pi^*A_i - \epsilon \cdot E)$ and $(\pi^*B_i  - \epsilon \cdot E)$ are 
 ample in $\widetilde X$.  
 
 Now, let $c_i\in \mathbb N$ be sufficiently large and divisible so that 
 $c_i\cdot (\pi^*A_i  - \epsilon E)$ and $c_i\cdot (\pi^*B_i -\epsilon \cdot E)$ are very ample.
 After pulling back we have
 
 $$
 (c_i\cdot m_i)\cdot \pi^*D_i   =  c_i\cdot(\pi^*A_i  -  \epsilon \cdot E)   -  c_i \cdot(\pi^*B_i  - \epsilon \cdot E).
 $$
 The original arguments of Kawamata as in the proof of~\cite[Thm.~4.1.10]{PAG1}---via the so-called 
 Bloch-Gieseker covering---now apply and we can construct a covering 
 $\sigma: Z\to X$ by taking roots. In particular there exists a line bundle $\sN$ on $Z$
 such that
 $$
 \sN^{\otimes(a_i\cdot c_i\cdot m_i)}  \cong  \sigma^*\O_{\widetilde X}\big( \pi^*(c_i\cdot m_i \cdot D_i) \big).
 $$
 Following the rest of the arguments of Kawamata as in~\cite[Prop.~4.1.12]{PAG1}
 we can then proceed to construct the morphism $\widetilde f$ as above
 with $H$ (the additional branch locus) being a general member of a very ample linear system 
 in $X$.
  \end{example}

 \subsection{Orbi-structures for log canonical spaces}
 \label{subsect:OrbiLc}
 In this section, we construct orbi-étale structures in codimension two for pairs satisfying the assumptions of Theorem~\ref{thm:main-ineq}. Later, this will make it possible to define a meaningful notion of Chern classes for such pairs.  \\
 
 But first, let us recall the following result, which is a consequence of the classification of dlt singularities 
 pairs in dimension two, see~\cite[Cor.~9.14]{GKKP} for more details. 
 
 \begin{propo}
 \label{prop:LcReduced}
 \emph{\textbf{(Dlt spaces admit orbi-étale structures in codimension 2).}}
 
 \noindent
 Let $(X,D)$ be a dlt pair with reduced $D$. There exists a Zariski open subset $X^\circ \subseteq X$
 with $\codim_X(X\backslash X^\circ)\geq 3$ over which $(X,D)$ admits an orbi-\'etale structure.
 
 \end{propo}
 
 \begin{proof}
 According to~\cite[Cor.~9.14]{GKKP}, there exists a Zariski open subset $X^1\subseteq X$ with 
 $\codim_X(X\backslash X^1)\geq 3$ such that
 $$
 D|_{X_1} \subset \bigcup_{i=1}^k U_i,
 $$
 where each $U_i$ is a Zariski open subset equipped with a finite, quasi-\'etale, surjective and 
 Galois morphism of quasi-projective varieties $f_i: X_i \to U_i$ and such that  
 $(X_i, f_i^*(D))$ is log-smooth. 
 
 On the other hand, as $X\backslash D$ is klt, it has only quotient singularities in codimension two, 
 cf.~\cite[Prop.~9.3]{GKKP}. Thanks to Artin's approximation for algebraic quotient singularities \cite[Cor.~2.6]{Artin69},
 it follows that there exists an open subset $X^2\subseteq (X\backslash D)$ with $\codim_{X\backslash D} ((X\backslash D) \backslash X^2) \geq 3$,
 admitting a Zariski open covering $U_{k+1},  \ldots, U_{k+l}$ satisfying the following properties. 
 
 \begin{enumerate}
 \item Each $U_i$, $i\in \{ k+1, \ldots, k+l   \}$, is equipped with an \'etale and 
 surjective morphism $\sigma_i:  \widetilde U_i  \to  U_i$.
 
 \item For each such $\widetilde U_i$ there exists a quasi-\'etale, Galois, surjective 
 morphism $g_i:  X_i  \to \widetilde U_i$ with $X_i$ smooth that is ramified only over the singular locus of $\widetilde U_i$.
 \end{enumerate}

 Now, for $i\in \{k+1, \ldots, k+ l \}$, let $f_i:= g_i \circ \sigma_i$ and define $X^\circ =  X^1\cap X^2$.
 The set $\{ (U_{\alpha}, f_{\alpha}, X_{\alpha}) \}_{\alpha\in I}$, $I= \{ 1, \ldots, k+l  \}$,
 now defines a orbi-\'etale structure over $X^\circ$.

 \end{proof}

\subsection{Two examples of singularities without orbi-structures}
\label{examples}
The very elementary examples below show how the restrictions on the singularities in Theorem~\ref{thm:main-ineq} cannot be removed if one wants to find a \textit{smooth} cover (in codimension two) that is adapted to the boundary divisor.\\

The first example below shows that if the pair $(X,D)$ is not dlt but merely log canonical, there is no hope in finding a smooth cover in codimension two, even if the coefficients are standard. 
 
  \begin{example} 
  \label{cusp}
  Set $X:=\mathbb C^2$, and let $C=\{y^2+x^3=0\}\subset \mathbb C^2$ be the cusp. Then it is well known that $\mathrm{lct}(X, C)= \frac 56$, which means that $(X,\frac 56 C)$ is lc but not klt. Denote by $Y:=\{t^6=y^2+x^3\}\subset \mathbb C^3$ the standard cover and set $C':=\{t=0\}$. The ramification formula can be written the following two ways 
$$K_Y \sim_{\Q} p^*(K_X+\frac 56 C) \quad \mbox{or} \quad K_Y+C' \sim_{\Q} p^*(K_X+C)$$
which shows that $Y$ is indeed lc but not klt and that $(Y,C')$ is not lc. (We also note that the singularity $(Y,0)$, called simple elliptic, is not ---as expected--- a finite quotient singularity, cf.~\cite[Thm. 3.6]{Kollar97}.)
\end{example}

\noindent
 The next example, in a similar vein as above, shows that if one does not require $D$ to have standard coefficients, then the ramified covers will in general \emph{not} be klt; an obstruction to much of the theory that will be developed in 
 Subsection~\ref{subsect:OrbiChern}, cf also Proposition~\ref{prop:LcReduced}.
 
   \begin{example} 
   \label{standard}
   Let $(X,D)$ be a pair with $X$ smooth, $D$ irreducible and reduced such that $\mathrm{lct}(X,D)<1$. Then for $m$ large enough, $(X,\frac 1mD)$ is klt; however, any cover $p:Y\to X$ ramified at order $m$ along $D$ will satisfy
   $$K_Y \sim_{\Q} p^*(K_X+\frac{m-1}{m}D)$$
   but for $m$ large enough, the pair $(X, \frac{m-1}{m}D)$ is not klt (or lc) anymore so that $Y$ has worse singularities than lc. Therefore $Y$ cannot be (locally) covered by a smooth variety, which as we will see in Subsection~\ref{subsect:OrbiChern} poses a major difficulty for defining orbifold Chern classes.
\end{example}

 \subsection{The approximation process}
 In the current section we establish a technical ingredient that enables us to reduce 
 the problem of establishing the Miyaoka-Yau inequality to the klt case. This is the content of Proposition~\ref{prop:OrbiLc},
 but first need the following variant of~\cite[Prop.~5.20]{KM} on the behaviour of singularities under certain 
 finite morphisms. We note that we will use the notation $(X,D)_{\reg}$ to denote the locus of $X_{\reg}$
 over which $D$ has simple normal crossing support.

 \begin{lemma}
 \label{lem:dlt}
 Let $Y$ be a quasi-projective variety and let $D_Y$ be an effective $\Q$-divisor such that $(Y,D_Y)$ is dlt. Let $f:X\to Y$ be a finite surjective morphism from a normal variety $X$ and let us define the effective divisor $D_X$ by the identity $K_X+D_X=f^*(K_Y+D_Y)$. Assume that $f^{-1}((Y,D_Y)_{\rm reg})\subset (X,D_X)_{\rm reg}$. Then $(X,D_X)$ is dlt. 
 \end{lemma}

\begin{proof}
We follow the proof and notations of \cite[Prop.~3.16]{Kollar97}. Let $Z_Y$ (resp. $Z_X$) be the complement of $(Y,D_Y)_{\rm reg}$ (resp. $(X,D_X)_{\rm reg}$) in $Y$ (resp. $X$). By the assumption made, one has $f(Z_X) \subset Z_Y$. For any birational proper morphism $\pi_Y:Y'\to Y$, one deduces a proper birational morphism $\pi_X:X'\to X$ by taking the normalization of the fiber product. As $(Y,D_Y)$ is dlt, it follows that for any prime divisor $E_Y\subset Y'$ such that $\pi_Y(E_Y)\subset Z_Y$, one has $a(E_Y,Y,D_Y)>-1$. Let us now consider a prime divisor $E_X\subset X'$ such that $\pi_X(E_X)\subset Z_X$, its image $E_Y$ on $Y'$ satisfies $\pi_Y(E_Y)\subset Z_Y$, and it the follows from the proof of \textit{loc. cit.} that $a(E_X,X,D_X)\ge a(E_Y,Y,D_Y)$, which concludes the proof. 
\end{proof}

 \begin{propo}\label{prop:OrbiLc}
 Let $(X,D)$ be a dlt pair with standard coefficients such that each component of $D$ is $\Q$-Cartier. Then, there exists an ample divisor $H$ such that
 the following condition is satisfied:
 
 \noindent
For all integers $m\geq 2$, the new pair $(X, D_m): = (X, D + \frac{1}{m}\cdot H)$ is a klt pair 
 admitting an orbi-\'etale structure in codimension two, up to a strictly adapted morphism.
 More precisely, there exists a strictly adapted morphism $f:Y\to (X, D- \lfloor D \rfloor)$
 with $Y$ having an orbi-\'etale structure $\{ (V_{\underline{\alpha}(m)} , f_{\underline{\alpha}(m)} , Y_{\underline{\alpha}(m)})  \}$ in codimension two with respect to the divisor $f^*(D_m)$.

 \end{propo}
 
 \begin{remark}
 The orbifold structure on $(X,D_m)$ is \textit{not} the obvious one but it is rather the one adapted to the decomposition~\eqref{decomp} below.
 \end{remark}
 
 \begin{proof} [Proof of Proposition~\ref{prop:OrbiLc}]
  Let $L$ be a very ample Cartier divisor and $A\in |L|$ a general member to be chosen later. 
 Define $H:=  A -\lfloor D \rfloor $. 
 Now, the divisor $D_m:=D+(1/m)\cdot H$ can be decomposed as:
 \begin{equation}
 \label{decomp}
 D_m : = \; \; \underbrace{(D - \lfloor D \rfloor)}_{D^{\orb}} \quad + 
 \quad (1-\frac{1}{m})\cdot  \lfloor D \rfloor \quad + \quad \frac 1m \cdot A.
 \end{equation}
 From here, we proceed in several steps. \\
 
 \noindent
 \textbf{Step 1.} \emph{$(X, D_m)$ is klt.} 
 
 \noindent
 Given $\pi:\widetilde  X\to X$ a log resolution of $(X, D)$, the linear system $|\pi^*L|$ on $\widetilde X$
 is basepoint free. By Bertini theorem, one can choose a sufficiently general member $\pi^*A$ of that system that contains no components of the exceptional divisor or the strict transform of $D$, and such that $\pi^*A+\mathrm{Exc}(\pi)+\pi^*D$ has snc support, cf.~\cite[9.1.9 \& 9.2.29]{PAG2}. As a consequence, the pair $(X, D+A)$ is lc and as $X$ is klt, it follows easily that $(X,D_m)$ is klt as soon as $m\ge 2$. More can be said: since $(X,D)$ is dlt, the proof of \cite[Lem.~5.17(2)]{KM} and Szabo's characterization of dlt pairs \cite[Thm.~2.44]{KM} easily imply that $(X,D+A)$ is dlt for $A$ general.\\
 
 \bigskip

 \noindent
 \textbf{Step 2.} \emph{Global cover adapted to $D^{\rm orb}$.}
 
 \noindent
 Following Example~\ref{ex:QFactorial}, we can find a morphism $f: Y \to (X, D^{\rm orb})$
 strictly adapted to $D^{\rm orb}$. By construction, $f$ is ramified  
 along a divisor a general member $H'$ of a very ample linear system 
 in $X$. For simplicity we assume that $f$ is totally ramified along $H'$.
 As such, the same arguments as in Step.~1
 show that $(X, D^{\rm orb} + H')$ is dlt and thus so is 
 $$
 \big(  X,  D^{\rm orb} + \frac{N- 1}{N}\cdot H' \big),
 $$
 where $N= \deg(f)$.
 From the ramification formula
 $$
 K_Y  +  f^*(\lfloor D \rfloor + A)\sim_{\mathbb Q} f^*(K_{X}  +  D + \frac{N-1}{N}\cdot H' +A),
 $$
 together with Lemma~\ref{lem:dlt}, it follows that $(Y, f^*(\lfloor D \rfloor + A))$ is dlt.
 
 \bigskip 
    \noindent
 \textbf{Step 3.} \emph{Covers adapted to $ (1-\frac{1}{m}) \cdotp \lfloor D\rfloor +\frac 1m \cdotp A$.}
 
 \noindent
 According to Proposition~\ref{prop:LcReduced}, there exists $Y^\circ \subseteq Y$ with
 $\codim_Y(Y\backslash Y^\circ)\geq 3$, over which we can
  find a finite family of (quasi-finite) quasi-\'etale covers $h_{\alpha}: W_{\alpha} \to V_{\alpha}$
 by quasi-projective varieties $W_{\alpha}$ factoring as follows:
 
 $$
 \xymatrix{
 W_{\alpha} \ar[rr]^{h_{\alpha}} \ar[dr]_{\text{Galois}}^{r_{\alpha}} && V_{\alpha}  \\
 & \widetilde V_{\alpha}  \ar[ur]^{\sigma_{\alpha}}_{\text{\'etale}}
 }
 $$
 
 and such that:
 
 \begin{enumerate}
  \item[$\bullet$] $Y^\circ = \bigcup_{\alpha} V_{\alpha}$.  
 \item[$\bullet$] For each $\alpha$, the pair $(W_{\alpha},h_{\alpha}^*f^*(\lfloor D\rfloor +A))$ is 
 log-smooth. 
 \item[$\bullet$] Each $r_{\alpha}$ is Galois and $\sigma_{\alpha}$ is \'etale.
 \end{enumerate}

 Now, following Example~\ref{ex:OrbiSmooth}, for each $W_{\alpha}$, one can find a finite 
 collection of finite maps
 $$
 g_{\alpha\beta}:  Y_{\alpha\beta} \to W_{\alpha},
 $$
 whose images cover $W_{\alpha}$, are branched exactly at order $m$ along 
 $h_{\alpha}^*\big( f^*(\lfloor D\rfloor +A)  \big)$. Moreover each 
 $g_{\alpha\beta}^*h_{\alpha}^*\big( f^*(\lfloor D \rfloor + A)  \big)$ is log-smooth.
 
 Now, define $f_{\alpha\beta}:= g_{\alpha\beta}\circ h_{\alpha}$. Noting that
 the branch locus of $g_{\alpha\beta}$ is equal to $h_{\alpha}^{-1}(\lfloor D\rfloor +A)$,
 from the constructions in Example~\ref{ex:QFactorial} it follows that along the 
 fibres of $(g_{\alpha\beta}\circ r_{\alpha})$ the ramification index is constant and
 therefore it is Galois.
 
 Let $\underline{\alpha}(m)$ denote the indexing pair $(\alpha, \beta)$ and define 
 $$
 W_{\underline{\alpha}(m)}  := { \rm Im} \big( f_{\underline{\alpha}(m)}  :  Y_{\underline{\alpha}(m)} \to W_{\alpha}\big),
 $$
 $$
 h_{\underline{\alpha}(m)} :=  h_{\alpha}|_{W_{\underline{\alpha}(m)}}, \; \; \text{and}
 $$
 $$
 V_{\underline{\alpha}(m)} :=  {\rm Im} \big( h_{\underline{\alpha}(m)}: W_{\underline{\alpha}(m)} \to V_{\alpha}  \big).
 $$
 Finally, set $D^{\red}_{W_{\underline{\alpha}(m)}}: = h^*_{\underline{\alpha}(m)}(f^*(\lfloor D \rfloor))$ and 
 $A_{W_{\underline{\alpha}(m)}}:= h^*_{\underline{\alpha}(m)} ( f^*A )$.
We summarize this construction in the following diagram. 

$$
 \begin{xymatrix}{
    Y_{\underline{\alpha}(m)} \ar@/^9mm/[rrrrrr]^{f_{\underline{\alpha}(m)}} 
    \ar[rrrrr]^{g_{\underline{\alpha}(m)}}_{\text{adapted to $(1-\frac 1m)D^{\red}_{W_{\underline{\alpha}(m)}}+\frac 1m A_{W_{\underline{\alpha}(m)}}$}} &&&&& W_{\underline{\alpha}(m)} 
    \ar[r]_{h_{\underline{\alpha}(m)}} & V_{\underline{\alpha}(m)}\subseteq Y  \ar[rrr]^{f}_{\text{adapted to $D^{\orb}$}} 
    &&& X.}
    \end{xymatrix} 
 $$

 The collection $\{ \big(  V_{\underline{\alpha}(m)} , f_{\underline{\alpha}(m)},  Y_{\underline{\alpha}(m)}  \big) \}$
 now defines, up to the strictly adapted morphism $f: Y\to X$, an orbi-\'etale structure for the pair 
 $(X, D_m)$.
 \end{proof}

 \subsection{Orbi-sheaves and Chern classes}
 \label{subsect:OrbiChern}
 In this subsection we introduce sheaves and Higgs sheaves associated to orbifold structures. 
 We will then define a notion of orbifold Chern classes. 
 
 \begin{defin}[Orbi-sheaves]
 \label{def:OrbiSheaf}
 Let $\mathcal C = \{ (U_{\alpha}, f_{\alpha}, X_{\alpha}) \}_{\alpha \in I}$ be an orbi-structure 
 on a given pair $(X,D)$. As in Definition~\ref{def:GlobalOrbi}, let $X_{\alpha\beta}$ 
 be the normalization 
 of the fibre product $X_{\alpha} \times_{X} X_{\beta}$ with naturally induced morphisms 
 $g_{\alpha\beta}: X_{\alpha\beta}\to X_{\alpha}$ and $g_{\beta\alpha}: X_{\alpha\beta}\to X_{\beta}$.
 We call a collection $\{ \F_{\alpha} \}_{\alpha}$ of 
 coherent sheaves of rank $r$
 on each $X_{\alpha}$ an orbi-sheaf on $(X,D)$ of rank $r$ with respect to $\mathcal C$,
 if the following conditions are verified. 
 
 \begin{enumerate}
 \item \label{compa} 
 There exists an isomorphism 
 $$
 g_{\alpha\beta}^*(\F_{\alpha})\cong g_{\beta\alpha}^*(\F_{\beta})
 $$
 of sheaves of $\O_{X_{\alpha\beta}}$-modules on $X_{\alpha\beta}$.
 
 \item The collection $\{\F_{\alpha}\}_{\alpha}$ 
 verifies further natural compatibility conditions over triple overlaps. 
 \end{enumerate}
 We denote this collection by $\F_{\mathcal C}$.
 We say that $\F_{\mathcal C}$ is torsion-free, reflexive or locally free orbi-sheaf,
 if each $\F_{\alpha}$ is torsion-free, reflexive or locally free, respectively. 
 \end{defin}

 \begin{remark} 
 
From the compatibility condition~\ref{compa}, with $\alpha=\beta$, it follows that 
each $\F_{\alpha}$ has a natural structure of a $G_{\alpha}$-sheaf, where $G_{\alpha}:= \Gal(X_{\alpha} / \widetilde U_{\alpha})$. 
 
 \end{remark}

 \begin{defin}[Orbi-subsheaves]\label{def:OrbiSubsheaf}
 Let $(X,D)$ be a pair with an orbi-structure $\mathcal C = \{ (U_{\alpha}, f_{\alpha}, X_{\alpha}) \}_{\alpha \in I}$.
 Let $\E_{\mathcal C}$ and $\F_{\mathcal C}$ be two orbi-sheaves
 with respect to $\mathcal C$. We say that $\F_{\mathcal C}$ is an orbi-subsheaf 
 of $\E_{\mathcal C}$, if, for each $\alpha$, we have the inclusion 
 $\F_{\alpha}\subseteq \E_{\alpha}$.

 \end{defin}
 
 One can also naturally define orbifold morphisms.
 
 \begin{defin}[Orbifold morphisms]\label{def:OrbiMorph}
 An orbifold morphism of two orbi-sheaves $\F_{\mathcal C}$ and $\G_{\mathcal C}$
 is a collection of  morphisms of sheaves of $\O_{X_{\alpha}}$-modules $\phi_{\alpha}: \F_{\alpha}\to \G_{\alpha}$
 that can be glued, that is, for every $\alpha$ and $\beta$, the diagram 
 
  $$
 \xymatrix{
 g_{\alpha\beta}^*\F_{\alpha} \ar[rr]^{g_{\alpha\beta}^* \phi_{\alpha}} \ar[d]^{\cong}  && g_{\alpha\beta}^*\G_{\alpha} \ar[d]^{\cong} \\
 g_{\beta\alpha}^*\F_{\beta} \ar[rr]^{g_{\beta\alpha}^*\phi_{\beta}}   && g_{\beta\alpha}^*\G_{\beta} 
 }
 $$
 commutes, where the vertical isomorphisms are the ones defined in Definition~\ref{def:OrbiSheaf}.

 \end{defin}

 \vspace{4 mm}
 
 \subsubsection{Higgs sheaves in the orbifold setting} 
 Higgs bundles are holomorphic bundles with compatible and integrable smooth 
 operators. To define orbifold Higgs sheaves we first recall the definition of the sheaf of $1$-forms
 adapted to a given orbifold structure.
 
\begin{notation}\label{not:adapt} Let $f: Y\to X$ be a morphism adapted to $D$, 
where $D=\sum d_i\cdot D_i$, $d_i=(1-\frac{b_i}{a_i})  \in (0,1]\cap \Q$.
Set $D_{f} = f^*(\lfloor D \rfloor)$.
For every prime component $D_i$ of $D-\lfloor D\rfloor$, let $\{D_Y^{ij}\}_{j(i)}$ be the 
collection of prime divisors that appear in $f^{-1}(D_i)$. We define new divisors in $Y$ by 
\begin{flalign}\label{not}
\widehat D_Y^{ij}:=b_i\cdot D_Y^{ij} 
\end{flalign}
\end{notation}

\begin{defin}[Orbifold (co)tangent sheaf]\label{def:adaptedsheaf} 
Let $(X,D)$ be a pair with a given orbifold structure 
$\mathcal C=\{ (U_{\alpha} , f_{\alpha} , X_{\alpha}  ) \}_{\alpha \in I}$.
Let $G_{\alpha}:=\Gal(X_{\alpha}/ \widetilde U_{\alpha})$.
For every $\alpha \in I$, define $D^{ij}_{X_{\alpha}}$ and $D_{f_{\alpha}}$ 
to be the divisors defined in 
Notation~\ref{not:adapt}.  
Define $\Omega_{(X_{\alpha}, f_{\alpha}, D)}^1$ to be the kernel of the sheaf morphism

$$
  \xymatrix{
    ((f_{\alpha})|_{X_{\alpha}})^{[*]}\bigl(\Omega_X^{[1]} (\log (\ulcorner D \urcorner))\bigr) \ar[rr] && \bigoplus \limits_{i,j(i)}\O_{{\widehat D_{X_{\alpha}}^{ij}}}
 }
  $$

\noindent
induced by the natural residue map.   
We define the orbi-cotangent sheaf $\Omega^{1}_{\mathcal C}$ of $(X,D)$ 
with respect to $\mathcal C$ to be the orbi-sheaf given by the collection of $G_{\alpha}$-sheaves 
$\{   \Omega_{(X_{\alpha}, f_{\alpha}, D)}^{1} \}_{\alpha \in I}$.
We define the orbifold tangent sheaf $\sT_{(\mathcal C, D)}$ by the collection $\{ \sT_{(X_{\alpha}, f_{\alpha}, D)} \}_{\alpha}$, 
where $\sT_{(X_{\alpha}, f_{\alpha}, D)}: =  (  \Omega^{1}_{(X_{\alpha}, f_{\alpha}, D)})^*$.
 \end{defin}
  
 By $\Omega^{[1]}_X$ and $f^{[*]}(\cdot)$ in Definition~\ref{def:adaptedsheaf} we denote the reflexive hull
 of $\Omega_X^1$ and $f^*$, respectively. 
 We note that when $(X_{\alpha}, \sum D^{ij}_{X_{\alpha}})$ is not log-smooth, through 
 coherent extensions, we can use Definition~\ref{def:adaptedsheaf} to define the orbifold cotangent sheaf
 as a reflexive sheaf on $X_{\alpha}$, which we can denote by $\Omega^{[1]}_{(X_{\alpha}, f_{\alpha}, D)}$.
 We refer to \cite[\S 3]{CKT} for local, explicit description of the orbi-cotangent sheaf in terms of differential forms with zeros and poles. Using this description it is not difficult to check that Definition~\ref{def:adaptedsheaf}
 indeed defines an orbifold sheaf. Let us briefly explain the case of the smooth orbi-\'etale structure 
 $\mathcal C=\{ (U_{\alpha}, f_{\alpha}, X_{\alpha}) \}$ for a log-smooth pair $(X,D)$ with
 $D = \sum (1- \frac{1}{a_i}) D_i$. The more general case will be similar. 
 
 According to Definition~\ref{def:adaptedsheaf}, for each $\alpha$ we have
 $$
 \Omega^1_{(X_{\alpha}, f_{\alpha}, D)} =  \Omega^1_{X_{\alpha}}.
 $$
 Let $X_{\alpha\beta}$ be the normalization of $X_{\alpha}\times_X X_{\beta}$
 with maps $f_{\alpha\beta}: X_{\alpha\beta} \to X_{\alpha}$ and $f_{\beta\alpha}: X_{\alpha\beta}\to X_{\beta}$.
  The condition~\ref{item:compatible} now guarantees that, as subsheaves of $\Omega^{1}_{X_{\alpha\beta}}$, 
  the two sheaves $f_{\alpha\beta}^*\Omega^1_{(X_{\alpha}, f_\alpha, D)}$ and 
  $f_{\beta\alpha}^*\Omega^1_{(X_{\beta}, f_{\beta}, D)}$ are isomorphic. The 
  compatibility condition over triple overlaps can be checked similarly.

  \
  
 \begin{defin}[Orbi-Higgs sheaves]
 Let $(X,D)$ be a pair equipped with an orbi-structure $\mathcal C$.
 We call an orbi-sheaf $\F_{\mathcal C}$ an orbi-Higgs sheaf, if 
 there is an orbi-sheaf morphism 
 $\theta_{\mathcal C}:  \F_{\mathcal C} \to  \F_{\mathcal C}\otimes  \Omega_{\mathcal C}^{1}$
 satisfying the integrability condition $\theta_{\mathcal C}\wedge \theta_{\mathcal C}=0$.
 An orbi-Higgs subsheaf is then defined to be an orbi-subsheaf that is invariant under $\theta_{\mathcal C}$.
 \end{defin}

 \subsubsection{Global covers associated to orbi-structures}
 \label{subsect:GlobalCover}
 In this section, we recall a construction due to Mumford \cite[\S 2]{MR717614} enabling to define Chern classes for varieties with quotient singularities. We refer to  \cite[\S 3.7]{GKPT15} for more details
 in the case of orbifold with no boundaries; the classical $\mathcal Q$-varieties. Let us note that in our case, the local covers will not be quasi-étale but this won't affect the general constructions. \\
 
 Let $(X,D)$ be a pair with an orbi-structure 
 $\mathcal C = \{ (U_{\alpha} , f_{\alpha}, X_{\alpha}  ) \}_{\alpha \in I}$. We consider a finite, Galois field extension of the function field $\mathbb C(X)$
 containing all the function fields $\mathbb C(X_{\alpha})$. Let $\widehat X_{\mathcal C}$ be the normalization 
 of $X$ in this field extension so that $X=\widehat X_{\mathcal C}/ G$ with Galois group $G$. Let 
 $f: \widehat X_{\mathcal C} \to X$ be the induced finite Galois morphism factoring through each $f_{\alpha}$: 
 
 $$
 \begin{tikzcd}
 \widehat X_{\mathcal C} \arrow{rrrr}{f}
 &&&& X \\
 \widehat X_{\alpha} \arrow[hookrightarrow]{u} \ar{rr}{q_{\alpha}}
 && X_{\alpha} \ar{rr}{f_{\alpha}}
 &&  U_{\alpha}  \arrow[hookrightarrow]{u}
 \end{tikzcd}
 $$
 where we have set $\widehat X_{\alpha}:=f^{-1}(U_{\alpha})$. By construction, $f|_{\widehat X_{\alpha}}$ factors through $f_{\alpha}$, so that there exists  $q_{\alpha}:\widehat X_{\alpha}\to X_{\alpha}$ such that $f|_{\widehat X_{\alpha}}=f_{\alpha} \circ q_{\alpha}$.
 We sometimes refer to $f$ as the \emph{Mumford cover} associated to $\mathcal C = \{ (U_{\alpha} , f_{\alpha}, X_{\alpha}  ) \}_{\alpha \in I}$. Furthermore, given any orbi-sheaf $\F_{\mathcal C}$ we can define a 
 $G$-sheaf $\widehat \F_{\mathcal C}$ on 
 $\widehat X_{\mathcal C}$ associated to $\mathcal C$ by 
 $$
 \widehat \F_{\mathcal C}|_{\widehat X_{\alpha}}=q_{\alpha}^*(\F_{\alpha}).
 $$

 \subsubsection{Orbi-Chern classes}\label{subsect:OrbiCherns}
 Let $(X,D)$ be a pair such that $X^\circ$, the maximal Zariski open subset of $X$ 
 over which $(X,D)$ admits an orbi-structure  $\mathcal C= \{  (U_{\alpha}^\circ, f_{\alpha}, X_{\alpha}^\circ   )  \}$, satisfies $\codim_X(X\backslash X^\circ)\geq 3$. 

 Let $f: \widehat X_{\mathcal C}^\circ \to
 X^\circ$ be the Mumford cover, defined in Subsection~\ref{subsect:GlobalCover}, 
 associated to the orbi-structure $\mathcal C$. 
 Let $G:=\Gal(\widehat X^\circ_{\mathcal C} / X^\circ)$. Following the notations introduced in Subsection~\ref{subsect:GlobalCover}, 
 let $\F^\circ_{\mathcal C}:=\{ \F^\circ_{\alpha}\}_{\alpha}$ be a coherent orbi-sheaf and $\widehat \F^\circ_{\mathcal C}$ the associated $G$-sheaf 
 on $\widehat X^\circ_{\mathcal C}$ defined by the equality $\widehat \F^\circ_{\mathcal C}|_{\widehat X_{\alpha}^\circ} = q_{\alpha}^*(\F^\circ_{\alpha})$.

As $\widehat X^{\circ}_{\mathcal C}$ is normal, it is Cohen-Macaulay in codimension $2$, so up to shrinking $X^{\circ}$ one can assume that $\widehat X^{\circ}_{\mathcal C}$ is Cohen-Macaulay. Combined with the smoothness of $X_{\alpha}^\circ$, this implies that the finite morphism $q_{\alpha}$ is flat (see, for example,~\cite[Obs.~3.5]{GKPT15}) and therefore the 
 finite resolution of each $\F^\circ_{\alpha}$ by locally free sheaves lifts. Thanks to \cite[Prop.~2.1]{MR717614}, the 
 $G$-sheaf $\widehat \F^\circ_{\mathcal C}$ on $\widehat X_{\mathcal C}^\circ$ has a finite, locally free resolution;
 in particular we can define its Chern classes $\widehat \F^\circ_{\mathcal C}$. We define the $i$th orbifold Chern class of 
 $\F^\circ_{\mathcal C}$ by:
 $$
 c_i(\F^\circ_{\mathcal C}):= \frac{1}{|G|}\cdot  \psi_{i} \bigl(  c_i(\widehat \F^\circ_{\mathcal C}) \bigr) \in 
 \A_{n-i} (X^\circ) \otimes \Q,
 $$
 where $\psi_{\bullet}$ is the canonical map  
 $$
 \psi_{\bullet}: \A^{\bullet}(\widehat X^\circ_{\mathcal C})^G \otimes \Q \to  \A_{n-\bullet}(X^\circ)\otimes \Q,
 $$
 defined by the cap product with $[\widehat X^\circ_{\mathcal C}]$ and pushforward.
 \footnote{As we are defining Chern classes as elements of $\A_{\bullet}(X^\circ)$ the appearance of 
 the factor $\frac{1}{|G|}$ is natural. More precisely, this choice is a reflection of the fact that
 the composition of maps 
 $\A_{\bullet}(X^\circ)  \to A^{n-\bullet}(\widehat X^\circ_{\mathcal C})^G \xrightarrow{\psi_{\bullet}} \A_{\bullet}(X^\circ)$ 
 is just multiplication 
 by $|G|$ (see~\cite[Lem.~3.5]{MR717614}). For the map $\A_{\bullet}(X^\circ)  \to A^{n-\bullet}(\widehat X^\circ_{\mathcal C})^G$ 
 we refer the reader to 
 \cite[Thm.~3.1]{MR717614} or Claim~\ref{claim:ring} in the appendix of this paper.}

 On the other hand, from 
 the localization sequence of Chow groups, we have
 $A_{n-i}(X^\circ)\cong A_{n-i}(X)$, whenever $ \codim_X(X\backslash X^\circ)>i$.
  As a result when $X$ is projective, and by using the homomorphism $\psi_{\bullet}$, the classes 
  $c_1(\F^\circ_{\mathcal C})$ 
  and $c_2(\F^\circ_{\mathcal C})$ are all well-defined as multilinear forms on 
  $\Neron^1(X)_{\mathbb Q}^{n-1}$and $\Neron^1(X)_{\mathbb Q}^{n-2}$, respectively.

  \begin{remark}
  \label{codim11}
  Recall from Remark~\ref{codim1} that any pair $(X,D)$ admits an orbi-structure $\mathcal C$ in codimension one. Now, if $\mathcal F^\circ_{\mathcal C}$ is a coherent orbi-sheaf with respect to $\mathcal C$, the construction above allows us to construct its first Chern class $c_1(\mathcal F^\circ_{\mathcal C}) \in \A_{n-1}(X)$.
  \end{remark}

 \begin{remark}
   It is clear from the above construction that Chern classes of orbi-sheaves can 
  be defined even in the absence of the Galois property for adapted morphisms in the orbifold structure 
  in Definition~\ref{def:LocalStructure} and the linearization property for such sheaves in 
  Definition~\ref{def:OrbiSheaf}. But as we will see next, such equivariance conditions  
  is necessary to have a meaningful notion of products for Chern classes. 
  \end{remark}

To see that we can define \emph{product} of Chern classes of orbi-sheaves in 
the Chow group $A(X^\circ)\otimes \Q$, we need to
equip this graded group with a ring structure compatible with $\psi_{\bullet}$.
To this end, we are going to assume that the orbifold structure 
$\mathcal C =\{ (U^\circ_{\alpha}, f_{\alpha}, X^\circ_{\alpha} ) \}$ satisfies the 
following condition: 

\begin{assumption}\label{assume}
In the setting of Definition~\ref{def:LocalStructure}, the (adapted) morphism $g_{\alpha}: X^\circ_{\alpha}
\to \widetilde U^\circ_{\alpha}$ factors as follows:

$$
\xymatrix{
X^\circ_{\alpha} \ar@/^10mm/[rrrrrr]^{g_{\alpha}}  \ar[rrrr]^{p_{\alpha}}_{\text{adapted}}  &&&&  W^\circ_{\alpha} \ar@/_7mm/[rrrr]_{p_{\alpha}'}
\ar@{->}[rr]^(.45){\text{quasi-\'etale}} && \widetilde U^\circ_{\alpha} 
  \ar[rr]^{\sigma_{\alpha}}_{\text{\'etale}} && U^\circ_{\alpha},
}
$$
where $W^\circ_{\alpha}$ is smooth and $p_{\alpha}$ is adapted (and flat).
\end{assumption}

Using the arguments of \cite[Thm. 3.1]{MR717614}, we show in Appendix~\S\ref{append} 
that $\psi_{\bullet}$ is a group isomorphism. A key consequence is, as in the 
case of \cite{MR717614}, that $A_\bullet(X^\circ)\otimes \Q$ can then be endowed with a compatible ring structure
(as long as Assumption~\ref{assume} is satisfied). As a result, 
as an element of $A_{n-k-\ell}(X^\circ)$, the product 
$c_k(\F^\circ_{\mathcal C}) \cdot c_\ell(\G^\circ_{\mathcal C})$ 
of two orbi-sheaves $\F^\circ_{\mathcal C}$ and $\G^\circ_{\mathcal C}$ is well-defined.
In particular, $c_1^2(\F^\circ_{\mathcal C})$ is a multilinear form on 
  $\Neron^1(X)_{\mathbb Q}^{n-2}$.

  To finish these preliminary discussions on the definition of orbi-Chern classes we now
  investigate a simple example where the orbifold data is given by a single chart. 
  
 \begin{example}
 \label{Ex:Easy}
 Let $(X, \frac{1}{m}D)$ be a log-smooth pair with $D$ being reduced and irreducible and let 
 $f: Y \to (X, \frac{1}{m}D)$ be an adapted morphism defining a smooth, orbi-\'etale structure 
 for $(X,\frac{1}{m}D)$. Assume that $f$ is totally ramified along $D$. Define 
 $D_Y:= f^{-1}D$. We are interested in computing $c_1$ and $c_2$ of 
 the orbi-sheaf $\{\O_{D_Y}\}$ in terms of (the numerical class of) $D$. 
 \begin{align*}
 c_1(\{ \O_{D_Y}  \})    & = \frac{1}{m} \cdot f_*( c_1(\O_{D_Y})\cap [Y])\\
                                  & = \frac{1}{m} \cdot f_*(D_Y)  \\
                                  &  = \frac{1}{m} \cdot D,               
 \end{align*}
 where by $D_Y$ (resp. $D$) we mean the class of $[D_Y]\in A_{n-1}(Y)$ (resp. $[D]\in A_{n-1}(X)$). Similarly we have:

 \begin{align*}
 c_2(\{ \O_{D_Y} \})     & = \frac{1}{m} \cdot f_*( c_2(\O_{D_Y})\cap [Y])\\
                                  & = \frac{1}{m} \cdot f_*(c_1(\mathcal O_Y(D_Y))^2\cap [Y])  \\
                                  &  = \frac{1}{m^2} D^2.
\end{align*}
  
 \end{example}

  \subsubsection{Slope and stability}
Having thus far established the definitions of orbi-sheaves and Chern classes, we can now define 
a notion of orbifold stability.

 \begin{defin}[Slope of orbi-sheaves]\label{def:Slope}
 Let $X$ be a normal projective variety and $D$ a $\Q$-effective divisor 
 with an orbi-structure $\mathcal C$ in codimension one. 
 Given an orbi-sheaf $\F_{\mathcal C}$
 we define its slope $\mu_{P}(\F_{\mathcal C})$ with respect to a nef divisor $P$ on $X$ by   
 $$
\mu_P(\F_{\mathcal C})= \frac{1}{\rank(\F_{\mathcal C})} \; c_1(\F_{\mathcal C}) \cdot [P]^{n-1},$$
where $[P]$ denotes the numerical class of $P$, cf.~Remark~\ref{codim11}.
 \end{defin}

 \begin{defin}[Stability of orbi-sheaves]
 In the setting of Definition~\ref{def:Slope}
 assume that the nef divisor $P$ verifies $P^{n-1}\not\equiv 0$. We say that a torsion-free orbi-sheaf 
 $\E_{\mathcal C}$ is semistable with respect to $P$, 
 if for every non-zero, torsion free orbi-subsheaf 
 $\F_{\mathcal C}\subset \E_{\mathcal C}$
 we have 
 \begin{equation}\label{eq:Slope}
 \mu_{P} ( \F_{\mathcal C}) \leq \mu_{P} ( \E_{\mathcal C}),
 \end{equation}
 In case of stability we require the inequality in~\eqref{eq:Slope} to be strict for any such subsheaf 
 $\F_{\mathcal C}$ with $\rank (\F_{\mathcal C} ) < \rank( \E_{\mathcal C} )$.
 
 \end{defin}

  \section{Behaviour of Chern classes and stability under change of orbi-structures}
 \label{section3-compatibility}
In this section, we investigate how Chern classes may change once we allow the orbifold structures to vary. Provided that change of structure is compatible in a sense to be defined, we show that such characteristic classes remain invariant
under change of orbifold structures. We will also establish similar results for the notion of slope stability.

\begin{set-up}\label{SetUp}
 Let $(X,D)$ be a pair equipped with two orbi-structures 
 $\mathcal C_1=\{  (U_{\alpha}, f_{\alpha}, X_{\alpha}  ) \}_{\alpha \in I}$ and 
 $\mathcal C_2=\{ (V_{\beta}, g_{\beta}, Y_{\beta}) \}_{\beta \in J}$. 
 Let $Z_{\alpha\beta}$ be 
 the normalization of the fibre 
 product $X_{\alpha}\times_{X} Y_{\beta}$ leading to the
 commutative diagram: 
 
 $$
  \xymatrix{
    Z_{\alpha\beta} \ar[rr]^{g_{\alpha\beta}}  \ar[d]_{f_{\alpha\beta}} && Y_{\beta} \ar[d]^{g_{\beta}} \\
      X_{\alpha} \ar[rr]^{f_{\alpha}}       && X.
  }
  $$
Having defined $h_{\alpha\beta}:= g_{\beta} \circ g_{\alpha\beta} $, assume that 
$\mathcal C':= \{ ( Z_{\alpha\beta} , h_{\alpha\beta} , U_{\alpha}\cap V_{\beta}) \}_{(\alpha,\beta)\in I\times J}$ is also an orbifold structure\footnote{
This can be easily guaranteed for example when $\mathcal C_1$ and $\mathcal C_2$ are strict (but this 
is not a necessary condition).}.
\end{set-up}

 \begin{defin}[Compatible orbi-sheaves]\label{def:Compatible}
 In the situation of Set-up~\ref{SetUp} let $\F_{\mathcal C_1}$ and $\G_{\mathcal C_2}$ be two locally free
 orbi-sheaves (with respect to $\mathcal C_1$ and $\mathcal C_2$ respectively) and define the two orbi-sheaves
 $\F_{\mathcal C'}:= \{ f_{\alpha\beta}^* \F_{\alpha} \}$ and $\G_{\mathcal C'}:= \{ g_{\alpha\beta}^* \G_{\beta} \}$
 (with respect to $\mathcal C'$)\footnote{The fact that $\F_{\mathcal C'}$ and $\G_{\mathcal C'}$ 
 are orbi-sheaves can be checked by a straightforward diagram chasing and the universal property of fiber
 products.}. We say $\F_{\mathcal C_1}$ and $\G_{\mathcal C_2}$ are compatible 
 if $\F_{\mathcal C'}$ and $\G_{\mathcal C'}$ are orbifold isomorphic 
 (see Definition~\ref{def:OrbiMorph}).


 \end{defin}

 \begin{example}
  \label{ex:compatible}
 \textbf{(Compatibility of orbi-cotangent sheaves)} 
 Let $\mathcal C_1$ and $\mathcal C_2$ be two strict orbi-structures for a pair $(X,D)$. 
 Then the orbi-cotangent sheaves $\Omega^{1}_{\mathcal C_1}$ and 
 $\Omega^{1}_{\mathcal C_2}$  are compatible. 
 
 \end{example}

 \begin{set-up}\label{SetUp2}
 In Set-up~\ref{SetUp}, let $\widehat X_{\mathcal C_1}$ and 
 $\widehat Y_{\mathcal C_2}$ be the global Mumford
 covers associated to $\mathcal C_1$ and $\mathcal C_2$, respectively 
 (see the construction in Subsection~\ref{subsect:GlobalCover}) and 
 set $G_1:= \Gal(\widehat X_{\mathcal C_1}/X)$ and 
 $G_2:=\Gal(\widehat Y_{\mathcal C_2}/X)$. 
 Let $\widehat Z$ be the normalization of $X$ in a Galois field extension of $\CC(X)$ containing all the 
 function field extensions $\CC(Z_{\alpha\beta})$, and with $G_{\widehat Z}:= \Gal(\widehat Z/ X)$, 
 resulting in the commutative diagram:
 
 $$
  \xymatrix{
  \widehat Z \ar[rrrrrr]^{\widetilde f_2} \ar[dd]_{\widetilde f_1} \ar[drrrr] &&&&&& \widehat Y_{\mathcal C_2} \ar[d] 
  \ar@/^5mm/[dd]^{f_{\widehat Y_{\mathcal C_2}}}  \\
  &&&& Z_{\alpha\beta} \ar[rr]^{g_{\alpha\beta}} \ar[d]^{f_{\alpha\beta}} \ar[drr]^{h_{\alpha\beta}} && Y_{\beta} \ar[d]_{g_{\beta}}  \\
  \widehat X_{\mathcal C_1} \ar[rrrr] \ar@/_5mm/[rrrrrr]_{f_{\widehat X_{\mathcal C_1}}} &&&& X_{\alpha} \ar[rr]^{f_{\alpha}} && X
.}
$$
\end{set-up}

In the next proposition, we show that Chern classes for compatible orbi-sheaves are well-defined.

 \begin{propo}
 \label{prop:InvariantChern}
 \emph{ \textbf{(Invariance of Chern classes for compatible orbi-sheaves)} }
 
 \noindent
 Let $X$ be a normal projective variety and $D$ a $\Q$-effective divisor such that 
 $(X,D)$ has a set of orbi-structures $J=\{ \mathcal C_i  \}$ in codimension $k\geq 1$. 
 Let $\{ \F_{\mathcal C_i} \}$ be a collection of compatible locally free (or reflexive) orbi-sheaves.
  Then, as multilinear forms on $\Neron^1(X)_{\Q}^{n-k}$, the Chern classes 
  $c_k$ of $\F_{\mathcal C_i}$
  are all equal.
  
  \noindent
 If additionally, the structures $\mathcal C_i$ satisfy Assumption~\ref{assume}, then the powers $c_\ell^p$ of the Chern classes  of $\F_{\mathcal C_i}$, seen as multilinear forms on $\Neron^1(X)_{\Q}^{n-\ell p}$, are all equal provided that $\ell p \le k$. 
  \end{propo}

 \begin{proof}
 Let $\mathcal C_1$, 
 $\mathcal C_2 \in J$.  
 It suffices to show that the desired Chern classes 
 for $\F_{\mathcal C_1}$ and $\F_{\mathcal C_2}$ coincide
 with the assumption that both are locally free.
Let $X^\circ$ be the maximal open subset of $X$ over which 
 $\mathcal C_1= \{ (U_{\alpha}^\circ, f_{\alpha}^\circ, X_{\alpha}^\circ) \}$ and 
 $\mathcal C_2=\{ (U_{\beta}^\circ, g_{\beta}^\circ, Y_{\beta}^\circ) \}$ are defined. 
 In the situation of Set-up~\ref{SetUp2}
 after replacing $X$ by $X^\circ$ we have 
 
 $$
 \xymatrix{
  \widehat Z^\circ  \ar[rrr]^{\widetilde f_2} \ar[d]_{\widetilde f_1}  &&&  \widehat Y^\circ_{\mathcal C_2} \ar[d]^{f_{\widehat Y^\circ_{\mathcal C_2}}}  \\
  \widehat X^\circ_{\mathcal C_1}  \ar[rrr]^{f_{\widehat X^\circ_{\mathcal C_1}}}    &&&  X^\circ,
 }
 $$
 where $\widehat Z^\circ$, $\widehat X_{\mathcal C_1}^\circ$ and $\widehat Y_{\mathcal C_2}^\circ$ 
 play the role of $\widehat Z$, $\widehat X_{\mathcal C_1}$ and 
 $\widehat Y_{\mathcal C_2}$ in Set-up~\ref{SetUp2}.

 The isomorphism  
 between the two locally-free sheaves $\widetilde f_1^*(\widehat \F_{\mathcal C_1})$ and 
 $\widetilde f_2^*(\widehat \F_{\mathcal C_2})$  
 implies that 
 $$
 c_k(\widetilde f_1^*\widehat \F_{\mathcal C_1}) = c_k(\widetilde f_2^* \widehat \F_{\mathcal C_2}) 
 \in  A^k(\widehat Z^\circ).
 $$ 
 The first part of proposition now follows by using the projection formula
 for Chern classes, cf.~\cite[Thm.~3.2]{Fulton98}, together with the commutativity of the above diagram. The second part is entirely similar granted that Assumption~\ref{assume} ensures that the map $\psi_{\bullet}$ from the construction in \S~\!\ref{subsect:OrbiCherns} is a ring morphism.  
 \end{proof}


 
 Next we show that, for compatible orbi-sheaves, semistability is also a well-defined notion.
 
 \begin{propo}
 \label{prop:IndepSS}
 \emph{ \textbf{(Independence of semistability for compatible orbi-sheaves).} }

 \noindent
 Let $\mathcal C_1$ and $\mathcal C_2$ be two orbi-structures over $X$ in codimension one.
 Let $\F_{\mathcal C_1}$ and $\F_{\mathcal C_2}$ be any two compatible locally free 
 orbi-sheaves. 
 Then, the orbi-sheaf $\F_{\mathcal C_1}$
 is semistable with respect to a nef divisor $P\subset X$, if and only if 
 $\F_{\mathcal C_2}$ is semistable. 
 
 \end{propo}

 \begin{proof}
 Let $X^\circ$ be the locus of $X$ over which $\mathcal C_1$ and $\mathcal C_2$ are both defined.
 We will follow the notations and constructions in Set-up~\ref{SetUp2} over $X^\circ$.

 Assume that $\F_{\mathcal C_1}$ is $P$-semistable and that
 $\F_{\mathcal C_2}$ is not $P$-semistable, as orbi-sheaves. 
 Then, $\widehat \F_{\mathcal C_2}$
 is not semistable. This implies that $\widetilde f_2^*(\widehat \F_{\mathcal C_2})$ 
 is not semistable. Then from the isomorphism 
 $$
 \widetilde f_1^*(\F_{\mathcal C_1})  \cong  \widetilde f_2^*(\F_{\mathcal C_2})
 $$
 it follows that 
 $\widetilde f_1^*(\widehat \F_{\mathcal C_1})$ is not semistable; let $\G\subset \widetilde f_1^*(\widehat \F_{\mathcal C_1})$ be a destabilizing subsheaf. 
 Now, according to~\cite[Thm.~4.2.15]{HL10}, there exists a subsheaf 
 $\sH\subset \widehat \F_{\mathcal C_1}$ such that 
 $$
 \G = \widetilde f_1^*(\sH).
 $$
 Noting that $\mu_P(\F_{\mathcal C_1}) = \mu_P(\F_{\mathcal C_2})$ 
 (thanks to Proposition~\ref{prop:InvariantChern}), it follows that $\sH$
 properly destabilizes $\widehat \F_{\mathcal C_1}$, contradicting 
 the semistability assumption on $\F_{\mathcal C_1}$.

 \end{proof}

 \subsection{The second Chern class of a pair $(X,D)$}
 \label{Chern}
 In this section, we explain how to define the second Chern class of a mildly singular pair $(X,D)$. 
 
 \begin{notation}[Orbi-Chern classes of pairs]\label{not:ChernPairs}
 Let $J=\{ \mathcal C_i \}$ be the collection of strict, orbi-structures for a given pair 
 $(X,D)$ satisfying Assumption~\ref{assume}. Assume that members of $J$ are \'etale orbi-structures so that $c_2(\Omega_{\mathcal C_i}^{1})$, 
 $c_1^2(\Omega_{\mathcal C_i}^{1})$ are 
 independent of the choice of $i$. We define $c_2(X, D):=c_2(\Omega_{\mathcal C_i}^{1})$
 and $c_1^2(X,D):=c_1^2(\Omega_{\mathcal C_i}^{1})$. 
 \end{notation}
 
 \noindent
One can define the square of the first Chern class as well as the second Chern class of any pair $(X,D)$ that satisfies the assumptions of Theorem~\ref{thm:main-ineq}. Indeed,
\begin{enumerate}
\item[$\bullet$] the proof of Proposition~\ref{prop:OrbiLc} shows that Assumption~\ref{assume} is satisfied up to a strictly adapted cover $f:Y\to X$. 
Therefore, when $D^{\orb}\neq 0$,
it is possible to define first and second Chern classes of the orbifold cotangent sheaf
as cycles on the adapted cover $Y$ and then, by an abuse of notation, define $c_i(X,D)$, for $i=1,2$ (resp. $c_1(X,D)^2$), to denote 
their cycle theoretic pushforward on $X$, divided by the degree of $f$. 
\item[$\bullet$] Furthermore, the orbi-étalité of these covers  on $Y$ coupled with 
a slight generalization of the proof of Proposition~\ref{prop:InvariantChern} guarantees that these definitions are independent of the 
choice of the orbifold structure on $Y$ and the adapted morphism $f:Y\to X$. In particular,
for $i=1,2$, the cycle $c_i(X,D)$ (resp. $c_1(X,D)^2$) in $X$
is well-defined.
\end{enumerate}

 
 \begin{example}[The log smooth reduced case]
 \label{loglisse0}
 If $(X,D)$ is log smooth and $D$ is reduced, then one can check from residue exact sequence that $c_2(X,D)= c_2(\Omega_X)+K_X\cdotp D + c_2(\mathcal O_D)$. If we write $D= \sum D_i$ its decomposition into irreducible components, then $c_2(\mathcal O_D)= \sum_{i}D_i^2+\sum_{i< j}D_i \cdotp D_j=D^2-\sum_{i< j}D_i\cdotp D_j$. 
 \end{example}
 
  \begin{example}[The general log smooth case]
  \label{loglisse}
 If $(X,D)$ is log smooth with coefficients in $[0,1]\cap \mathbb Q$ and $D=\sum(1-\frac {b_i}{a_i})D_i$ is its decomposition into irreducible components, then $(X,D)$ admits an orbi-étale structure (defined over the whole $X$). In particular, one can define $c_1(X,D)$ and $c_2(X,D)$; let us check that one recovers the classical classes, cf e.g. \cite{Tian94}. 
 
 \noindent
 Given a (local) strict orbi-étale cover $f:Y\to (X,D)$ 
of degree $N:=\prod_i a_i$, we have an exact sequence
 \begin{equation}
 \label{exseq}
 0 \longrightarrow \Omega^{1}_{(Y,f,D)} \longrightarrow f^*\Omega^1_{X}(\log \, \ulcorner D \urcorner) \longrightarrow 
 \bigoplus_{i,k(i)} \mathcal O_{b_i  D_Y^{ik}}\longrightarrow 0
 \end{equation}
 where $D_Y^{ik(i)}$ is the collection of prime divisors defined by $f^*D_i$ and the sum runs over 
 all indices $i$ and $k(i)$ such that $a_i \neq + \infty$. As $c_1(\mathcal O_{b_i  D_Y^{ik}})\cap [Y]=b_i[ D_Y^{ik}]$ and 
 $f_* D_Y^{ik}=\frac N{a_i} D_i$, we get 

 \begin{align*}
 c_1(X,D)= & \frac 1N f_*\Big(c_1(\Omega^{1}_{(Y,f,D)}) \cap [Y]\Big)\\
 =& c_1\big(\Omega_{X}^1(\log \, \ulcorner D \urcorner)\big)-\sum_{i, k(i)} \frac {b_i}N f_*( D_Y^{ik}\cap [Y])\\
 =&c_1(K_X+D).
 \end{align*}
  For the computation of $c_2(X,D)$, we will need the identity  
\begin{equation}
\label{eq:c2}
c_2\Big( \bigoplus_i \mathcal O_{b_i  D_Y^{ik}}\Big) =\sum_{i,k}b_i^2 c_1(\mathcal O_Y( D_Y^{ik}))^2+ 
\sum_{\substack{i,j,k,l \\ i\leq  j , k<l}} b_ib_j c_1(\mathcal O_Y( D_Y^{ik(i)}))\cdotp c_1(\mathcal O_Y( D_Y^{jl(j)})).
\end{equation}
along with the standard formula
$$c_2(\Omega_{X}^1(\log \, \ulcorner D \urcorner))=c_2(\Omega_X)+ c_1(K_X)\cdotp \sum D_i+ \sum_{i< j}D_i\cdotp D_j +\sum_i D_i^2,$$
where as usual we don't distinguish between $D_i$ and its class $[D_i]\in \A_{n-1}(X)$. 
Relying on the exact sequence~\eqref{exseq} and the identity 
$$f_*\Big( \sum_{k,l}  c_1(\mathcal O_Y( D_Y^{ik}))\cdot c_1(\mathcal O_Y( D_Y^{jl})) \cap [Y]\Big)=\frac{N}{a_ia_j}D_i\cdot D_j,$$ 
we get
{\footnotesize
 \begin{align*}
 c_2(X,D)= & \frac 1N f_*\Big(c_2(\Omega^{1}_{(Y,f,D)}) \cap [Y]\Big)-\frac 1N f_*\Big(c_2( \bigoplus_{i; a_i\neq+\infty} \mathcal O_{b_i  D_Y^{ik}})\cap [Y]\Big)  -\sum_i c_1(K_X+ D ) \cdot \frac{b_i}{a_i}D_i \\
 =& c_2\big(\Omega^1_{X}(\log \, \ulcorner D \urcorner)\big)- \sum_i \frac {b_i^2}{a_i^2} D_i^2-\sum_{i<j} \frac{b_ib_j}{a_ia_j}  D_i\cdot  D_j -\sum_i c_1(K_X) \cdot \frac{b_i}{a_i}D_i\\
 &-\sum_i \big(1-\frac{b_i}{a_i} \big)\frac{b_i}{a_i} D_i^2-\sum_{i<j}  \big(\frac{b_i}{a_i}+\frac{b_j}{a_j}-2 \frac{b_ib_j}{a_ia_j}\big)D_i\cdot D_j\\
 =&c_2(\Omega_X^1)+c_1(K_X)\cdot D+\sum_i \big(1-\frac{b_i}{a_i}\big)D_i^2+\sum_{i<j}\big(1-\frac{b_i}{a_i}\big)\big(1-\frac{b_j}{a_j}\big)D_i\cdot D_j.
 \end{align*}  
 }
 Here, $i,j$ range among all indices (not just those with $a_i\neq +\infty$) unless specified otherwise.

 \end{example}
 
 The next example shows that unlike $c_1(X,D)$, the second Chern class $c_2(X,D)$ does not only depend on the class of linear equivalence of $D$, so that one needs to be very careful when one modifies the divisor.
 
  \begin{example}
Let $X=\mathbb P^2$, let $x\in X$, let $\pi:Y=\mathrm{Bl}_xX\to X$, and let $E\subset Y$ the exceptional divisor. Choose $H_1,H_2$ two hyperplanes such that $x\in H_1$ and $x\notin H_2$. Let us denote by $H_1'$ (resp. $H_2'$) the strict transform of $H_1$ (resp. $H_2$) by $\pi$. Then one has $\pi^*H_1=H_1'+E$ and $\pi^*H_2=H_2'$. 
As $\pi^*H_2$ and $\pi^*H_1$ are linearly equivalent, the formula in Example \ref{loglisse0} shows that $c_2(Y,\pi^*H_2)-c_2(Y,\pi^*H_1)=H_1'\cdotp E=1$.
 \end{example}  
 
 \subsection{Continuity of Chern numbers}
 The following intuitive result shows that the approximation procedure introduced in Proposition~\ref{prop:OrbiLc} will not affect the properties of Chern classes. More precisely, let $(X,D_m)$ be the klt pair from Proposition~\ref{prop:OrbiLc} with its orbifold structure $\mathcal C_{\underline \alpha(m)}$ in codimension $2$ (which is not the obvious structure associated to $D_m=D+\frac 1m H$). Then we have:
 
 \begin{propo}
 \label{prop:continuity}
 \emph{ \textbf{(Continuity of orbifold intersection numbers).} }
With the notations of Proposition \ref{prop:OrbiLc}, the orbifold structure $\mathcal C_{\underline \alpha(m)}$ for pair $(X,D_m)$ satisfies:
\begin{eqnarray*}
\lim_{m\to +\infty}c_1(X,D_m)&=&c_1(X,D) \\
\lim_{m\to +\infty}c_2(X,D_m)&=&c_2(X,D)
\end{eqnarray*}
as multilinear forms on $\mathrm N^1(X)_{\mathbb Q}^{n-1}$ and $({\rm Nef}(X)_{\mathbb Q})^{n-2}$, respectively.
 \end{propo}

 \begin{proof}
 Recall that $D_m=D+\frac 1m H$, and as $c_1(X, D_m)=c_1(K_X+D_m)$ (see, for instance,~\cite[Cor. 3.9]{CKT}), the first identity follows. 
 
 To prove the second equality, we will assume that $D$ is reduced; as $D^{\orb}$ is independent of $m$, the fractional 
 case follows from identical arguments. 
 We shall utilize the 
 orbi-\'etale structure constructed in Proposition~\ref{prop:OrbiLc} and we will follow the notations introduced 
 in this construction with $U_{\underline{\alpha}(m)}= V_{\underline{\alpha}(m)}$ 
 (see the proof of Proposition~\ref{prop:OrbiLc}). 
 Let us also decompose $g_{\underline{\alpha}(m)}$
 into orbi-\'etale morphisms
 $t_{\underline{\alpha}(m)}: X_{\underline{\alpha}(m)}  \to  (W_{\underline{\alpha}(m)}, \frac{1}{m}A_{W_{\underline{\alpha}(m)}})$ 
 and $u_{\underline{\alpha}(m)}: Y_{\underline{\alpha}(m)} \to 
 (X_{\underline{\alpha}(m)}, (1-1/m)D_{X_{\underline{\alpha}(m)}})$:
 
  $$
 \begin{xymatrix}{
    Y_{\underline{\alpha}(m)} \ar@/^12mm/[rrrrrrrr]^{f_{\underline{\alpha}(m)}}   \ar@/^6mm/[rrrrrr]^{g_{\underline{\alpha}(m)}} \ar[rrr]^{u_{\underline{\alpha}(m)}} &&& X_{\underline{\alpha}(m)}  \ar@/^6mm/[rrrrr]^{d_{\underline{\alpha}(m)}}
     \ar[rrr]^{t_{\underline{\alpha}(m)}} &&& W_{\underline{\alpha}(m)} \ar[rr]^{h_{\underline{\alpha}(m)}} && U_{\underline{\alpha}(m)},
     }
    \end{xymatrix} 
 $$ 
 where $A_{W_{\underline{\alpha}(m)}}:= h_{\underline{\alpha}(m)}^{-1}(A)$ and $D_{X_{\underline{\alpha}(m)}}:=d_{\underline{\alpha}(m)}^{-1}(D)$. 
 Define $d_{\underline{\alpha}(m)}:= t_{\underline{\alpha}(m)}\circ h_{\underline{\alpha}(m)}$.
 By Proposition~\ref{prop:OrbiLc}
 we know that $f_{\underline{\alpha}(m)}^*(D)$ is snc and therefore we have: 
 
 \begin{equation*}
 0 \to  \Omega^1_{(Y_{\underline{\alpha}(m)}, f_{\underline{\alpha}(m)},  D_m )}  \to 
   u_{\underline{\alpha}(m)}^*\big( \Omega^1_{(X_{\underline{\alpha}(m)}, d_{\underline{\alpha}(m)}, \frac{1}{m}A)} \log (d_{\underline{\alpha}(m)}^*D) \big) \to \bigoplus \O_{D_{Y_{\underline{\alpha}(m)}}^{ij}}  \to 0,
 \end{equation*}
 where $\{D_{Y_{\underline{\alpha}(m)}^{ij}}\}_j$ are irreducible components of $f_{\underline{\alpha}(m)}^*(D_i)$.
 Let $\F_{\mathcal C_{\underline{\alpha}(m)}}$ and $\G_{\mathcal C_{\underline{\alpha}(m)}}$ be 
 the two orbi-sheaves associated to $u_{\underline{\alpha}(m)}^*\big( \Omega^1_{(X_{\underline{\alpha}(m)}, d_{\underline{\alpha}(m)}, \frac{1}{m}A)} \log (d_{\underline{\alpha}(m)}^*D) \big)$ and $\bigoplus \O_{D_{Y_{\underline{\alpha}(m)}}^{ij}} $, 
 respectively, so that 
 
 \begin{equation}
 \label{eq:mond}
 c_2(X, D_m) = c_2(\F_{\mathcal C_{\underline \alpha(m)}}) - \frac{1}{m} \big( (K_X+D_m)\cdot D \big)
         - c_2(\G_{\mathcal C_{\underline \alpha(m)}})
 \end{equation}
 (see Example~\ref{Ex:Easy}). 
 From the equality~\eqref{eq:mond} it is clear that it suffices to prove the following two claims.

 \begin{claim}\label{claim:big1}
 $\lim _{m\to \infty} c_2(\F_{\mathcal C_{\underline \alpha(m)}})= c_2(X,D)$.
 \end{claim}
 
 \begin{claim}\label{claim:big2}
 $\lim_{m\to \infty} c_2(\G_{\mathcal C_{\underline \alpha (m)}}) =0$.
 \end{claim}
 
\begin{proof}[Proof of Claim~\ref{claim:big1}]
 We use the exact sequence 
 $$
 0 \to \Omega^1_{(X_{\underline \alpha (m)}, d_{\underline \alpha (m)}, \frac{1}{m}A )} \to 
    \Omega^1_{(X_{\underline \alpha (m)}, d_{\underline \alpha (m)} , \frac{1}{m} A)} \log (d_{\underline \alpha (m)}^*D) 
       \to  \bigoplus \O_{D_{X_{\underline \alpha (m)}}^{ij} } \to 0 ,
    $$
where $D_{X_{\underline \alpha (m)}}^{ij}$ are irreducible components of $d_{\underline \alpha (m)}^*(D_i)$. 
From this sequence it follows that the equality  
$$
c_2(\F_{\mathcal C_{\underline \alpha (m)}}) = c_2(X, \frac{1}{m} A) + c_2(\sH_{\mathcal C_{\underline \alpha (m)}})
     + (K_X+ \frac{1}{m}A)\cdot D,
$$
holds, where 
$\sH_{\mathcal C_{\underline \alpha (m)}}$ is the orbi-sheaf associated to $\bigoplus \O_{D_{X_{\underline \alpha (m)}}^{ij}}$. 
Next, we use the fact that $t_{\underline{\alpha}(m)}$ is flat and unbranched at the generic point of $D$
together with standard Chern class calculations to show that 
$$
c_2(\sH_{\mathcal C_{\underline{\alpha}(m)}}) =  c_2( \sH'_{\mathcal C_{\underline{\alpha}(m)}} ),
$$
where $\sH'_{\mathcal C_{\underline{\alpha}(m)}}$ is the orbi-sheaf associated to $\bigoplus_i \O_{D^i_{W_{\underline{\alpha}(m)}}}$, with 
$ D_{W_{\underline{\alpha}(m)}} := h_{\underline{\alpha}(m)}^*(D) =
 \sum D^i_{W_{\underline{\alpha}(m)}}$.

On the other hand, we have
$$
c_2(X)  +  K_X\cdot D +  c_2(\sH'_{\mathcal C_{\underline{\alpha}(m)}}) = c_2(X,D).
$$

Therefore, to prove Claim~\ref{claim:big1} it suffices to show the following assertion. 

\begin{subclaim}\label{subclaim} 
$\lim_{m\to \infty} c_2(X, \frac{1}{m} A) = c_2(X).$ 

\end{subclaim}

\noindent
\emph{Proof of Subclaim~\ref{subclaim}.} We use the exact sequence 
\begin{equation}\label{exactly}
 0 \to \Omega^1_{(X_{\underline{\alpha}(m)}, t_{\underline{\alpha}(m)}, \frac{1}{m} A )}
\to  d^{[*]}_{\underline{\alpha}(m)} \Omega^{[1]}_{U_{\underline{\alpha}(m)}} \log (A)
\to \bigoplus_i \O_{\widehat{A}^i_{X_{\underline{\alpha}(m)}}} \to 0,
\end{equation}
where $\widehat{A}^i_{X_{\underline{\alpha}(m)}} = (m-1)A^i_{X_{\underline{\alpha}(m)}}$, 
$A^i_{X_{\underline{\alpha}(m)}}$ being the irreducible components of 
$d^{-1}_{\underline{\alpha}(m)}(A)$.

Let $\sA_{\mathcal C_{\underline{\alpha}(m)}}$ be the orbi-sheaf associated to 
$\bigoplus \O_{\widehat A^i_{X_{\underline{\alpha}(m)}}}$. 
Using the exact sequence (\ref{exactly})
we have
$$
c_2 (X, \frac{1}{m} A) = c_2(X) + K_X\cdot A + c_2(\sA_{\mathcal C_{\underline{\alpha}(m)}})
- c_2(\sA_{\mathcal C_{\underline{\alpha}(m)}}) - (K_X+\frac{1}{m} A)\cdot \big( \frac{m-1}{m} \big)\cdot A.
$$
We can now establish Subclaim~\ref{subclaim} by taking $m\to \infty$.

This finishes the proof of Claim~\ref{claim:big1}.
\end{proof}

\begin{proof}[Proof of Claim~\ref{claim:big2}]

Chern class calculations in Example~\ref{loglisse},  Equation~\ref{eq:c2},
with $a_i= m$ and $b_i=1$ show that 

\begin{equation}\label{eq:finally}
c_2(\G_{\mathcal C_{\underline{\alpha}(m)}}) \leq \frac{1}{m^2} \cdot D^2,
\end{equation}
as multilinear forms on $\rm Nef(X)_{\mathbb Q}^{n-2}$.
Now, as $m\to \infty$, the right hand side of (\ref{eq:finally}) goes to zero and thus so
does $c_2(\G_{\mathcal C_{\underline{\alpha}(m)}})$.
\end{proof}

The proof of Proposition~\ref{prop:continuity} is now complete.
\end{proof}

 \section{Semistability of the orbifold tangent sheaf}
\label{section4-semistability}

\subsection{Metrics with conic and cusp singularities}
\label{mixed}
Given numbers $d_1, \ldots, d_k \in (0,1)$ and a number $k\le p \le n$ one can consider for $\mathbb D= \{z\in \mathbb C: |z|<1\}$ the following model Kähler metric on $\mathbb (\mathbb D^*)^{p}\times \mathbb D^{n-p}$:
\begin{equation}
\label{mod}
\om_{\rm mod}= \sum_{j=1}^k\frac{idz_j \wedge d\bar z_j}{|z_j|^{2d_j}}+\sum_{j=k+1}^p \frac{idz_j \wedge d\bar z_j}{|z_j|^{2}\log^2|z_j|^2}+\sum_{j>p}i dz_j \wedge d\bar z_j
\end{equation}
The metric $\om_{\rm mod}$ is called the standard metric with mixed conic and cusp singularities along $D:=\sum_{j=1}^kd_j[z_j=0]+\sum_{j=k+1}^p [z_j=0]$. It has cone singularities with cone angles $2\pi (1-d_j)$ along $[z_j=0]$ for $1\le j \le k$ and cusp singularities along $[z_j=0]$ for $k+1\le j \le p$. \\

Now, if $X$ is a Kähler manifold and $D= \sum_i d_i D_i$ a divisor with simple normal crossings support and coefficients $d_i\in [0, 1]$, one says that a Kähler metric $\om$ on $X \smallsetminus \supp(D)$ has mixed conic and cusp singularities along $D$ if for any $x\in X$ and any open set $\Omega \ni x$ endowed with local holomorphic coordinates $(z_1, \ldots, z_n)$ such that the pair $(\Omega, D)$ is isomorphic to $(\mathbb (\mathbb D^*)^{p}\times \mathbb D^{n-p}, \sum_{j=1}^kd_j[z_j=0]+\sum_{j=k+1}^p [z_j=0])$, there exists a constant $C>0$ such that under that isomorphism, one has 
$$C^{-1}\om_{\rm mod} \le \om \le C\om_{\rm mod}$$
on $\Omega \smallsetminus \supp(D)$. \\

Given a compact Kähler manifold $X$ and a divisor $D=\sum_{j=1}^p d_j D_j$ with coefficients $d_j\in [0, 1]$ and simple normal crossings support, it is easy to give examples of such metrics. Moreover, if the real cohomology class $c_1(K_X+D)\in H^2(X, \mathbb R)$ is assumed to be Kähler (that is, it contains a Kähler metric) then there exists a unique Kähler metric $\om$  on $X\smallsetminus \supp(D)$ with mixed conic and cusp singularities along $D$ such that $\Ric \om = -\om$. This is the content of \cite{G12}, \cite[Thm. 6.3]{GP}. More precisely, it is shown there that any (weak) solution $\om_{X}+\ddc \vp$ of a Monge-Ampère equation of the type
$$(\om_X+\ddc \vp)^n= \frac{e^{\vp+F}}{\prod_{j=1}^p |\sigma_j|^{2d_j}}\, \om_X^n$$
has mixed conic and cusp singularities along $D$. Here, $\om_X$ is a reference Kähler form on $X$, $F\in \mathcal C^{\infty}(X)$, and $\sigma_j$ is the canonical section of $D_j$, measured with respect to an arbitrary smooth hermitian metric $|\cdotp |$ on $\mathcal O_X(D_j)$ and $\vp \in \mathcal E(X, \om_X)$ is the unknown function, cf \cite{GZ07} for the definition of the latter functional space.   \\

Let us conclude this section by explaining how the model metric $\om_{\rm mod}$ from \eqref{mod} pulls-back to ramified covers along $D$ when $D$ has fractional coefficients. More precisely, if one writes $d_j = 1-b_j/a_j$ and if  $f_0: \mathbb D^n \to \mathbb D^n$ is defined by $f_0(w_1, \ldots, w_n)= (w_1^{a_1}, \ldots, w_k^{a_k}, w_{k+1}, \ldots, w_n)$, then 
$$f_0^*\om_{\rm mod}= \sum_{j=1}^k |z_j|^{2(b_j-1)}i dz_j \wedge d\bar z_j+\sum_{j=k+1}^p \frac{idz_j \wedge d\bar z_j}{|z_j|^{2}\log^2|z_j|^2}+\sum_{j>p}i dz_j \wedge d\bar z_j$$
In particular, in the case where $D$ has standard coefficients (that is, $b_j=1$) and $\lfloor D \rfloor =0$, $f_0^*\om_{\rm mod}$ is equal to the euclidian metric. One sometimes say that in that case, $\om_{\rm mod}$ is smooth \textit{in the orbifold sense}. 

\noindent
In general though, $f_0^*\om_{\rm mod}$ will have zeros (and not poles anymore) along $\sum_{j=1}^kd_j[z_j=0]$. A useful observation is that  $f_0^*\om_{\rm mod}$ induces a continuous hermitian semipositive metric on the trivial vector bundle on $\mathbb D^n$ generated by $$z_1^{1-b_1}\frac{\d}{\d z_1} , \ldots, z_k^{1-b_k}\frac{\d}{\d z_k}, z_{k+1}\frac{\d}{\d z_{k+1}}, \ldots, z_{p}\frac{\d}{\d z_{p}},\frac{\d}{\d z_{p+1}}, \ldots, \frac{\d}{\d z_{n}}$$
which is degenerate precisely along $f_0^*\lfloor D \rfloor$. \\

\noindent
Note that one has a similar phenomenon if instead of choosing this very particular cover, one chooses a cover of the form $w \mapsto (w_1^{N}, \ldots, w_k^{N}, w_{k+1}, \ldots, w_n)$ where $N$ is divisible by any of the $a_j$'s. In that case, the pull-back of $\om_{\rm mod}$ by $\gamma$ is equal to $ \sum_{j=1}^k |z_j|^{2(Nb_j/a_j-1)}i dz_j \wedge d\bar z_j+\sum_{j=k+1}^p \frac{idz_j \wedge d\bar z_j}{|z_j|^{2}\log^2|z_j|^2}+\sum_{j>p}i dz_j \wedge d\bar z_j$.

\subsection{Setting}
\label{setting}
Let $(X,D)$ be a $n$-dimensional projective log canonical pair.
Let us set as before $D = \sum_{i=1}^r (1-\frac{b_i}{a_i}) D_i$ for some positive integers $a_i,b_i$ satisfying $b_i<a_i$ and $(a_i,b_i)=1$. One allows the possibility $a_i=+\infty$. One considers a strong log resolution $\pi:\wt X\to X$ of the pair $(X,D)$; one can write $\pi^*D= \wt D+\mbox{(exc. div.)}$ where $\wt D = \sum (1-b_i/a_i) \wt D_i$ is the strict transform of $D$. One has:
\begin{equation}
\label{canonical}
K_{\wt X}+\wt D=\pi^*(K_X+D)+E
\end{equation}
where $E=\sum c_j E_j$ is a $\pi$-exceptional divisor with coefficients $c_j \ge -1$. Finally, let us set $N= \mathrm{lcm} \{a_i, 1\le i \le r\} $ and choose a sufficiently ample divisor $H$ on $\wt X$; Kawamata's construction allows us to get a cover adapted to $(\wt X, \wt D)$. More precisely, one can find is a finite morphism $f: \wt Y\to \wt X$ which is a ramified Galois cover of group $G$ and it satisfies the following properties:
\begin{enumerate}
\item $f$ is étale over the complement of $\sum_{a_i<+\infty}\wt D_i+H$;
\item The support of $f^*(\wt D+H)$ has simple normal crossings;
\item Near any point $y_0\in \wt Y$, there exist a $G$-invariant open set $\Omega_0\ni y_0$, a system of coordinates $(w_k)$ centered at $y_0$, a system of coordinates $(z_k)$ near $f(y_0)$ and an integer $p=p(y_0)$ such that with respect to these coordinates, the map $f$ can be locally expressed as
$$f(w_1, \ldots, w_n)=(w_1^N, \ldots, w_p^N, w_{p+1}, \ldots, w_n)$$
where for each $1\le k \le p$, $(z_k=0)$ is a local equation (near $\gamma(y_0)$) of one of the components of $\sum_{a_i<+\infty}\wt D_i+H$.\\
\end{enumerate}

\noindent
Note that we have chosen that particular cover (with equal ramification index along all the divisors) only to simplify the notations, as any smooth adapted cover would have worked for what follows. Let us set 
$$D'=f^*\wt D= \sum_{a_i<+\infty} \left(N- \frac{Nb_i}{a_i}\right) \cdotp D_i'+\sum_{a_i=+\infty}D_i' $$ and $f^*H=NH'$.  With these notations, the divisors $D_i'$'s and $H$ are reduced. With these notations, the ramification divisor of $f$ becomes $(N-1)\sum D'_i+(N-1)H'$.  By the ramification formula, one gets:
$$K_{\wt Y}+ \sum\left(1-\frac{Nb_i}{a_i}\right)D_i'= f^*\left(K_{\wt X}+\wt D+\left(1-\frac{1}{N}\right)H\right)$$
or equivalently
$$K_{\wt Y}+ \sum\left(1-\frac{Nb_i}{a_i}\right)D_i'+(1-N)H'-f^*E= f^*\pi^*\left(K_X+D\right)$$
Here, the index $i$ varies over all possible indices (and not just those for which $a_i\neq+\infty)$. To lighten notation, let us set $B=\sum e_i B_i$ where $B_i$ is either one of the divisors $D_i'$  in which case $e_i:=1-Nb_i/a_i$ or $B_i=H'$ in which case $e_i:=1-N$. One gets a divisor $B$ with \textit{integral} coefficients and simple normal crossings support which satisfies 
\[K_{\wt Y}+B=f^*\pi^*\left(K_X+D\right)+f^*E.\] 
Note that neither $B$ nor $-B$ is effective in general (unless either $\lceil D \rceil = 0$ in which case $B \le 0$ or $D$ is reduced in which case $f=\mathrm{Id}_{\wt X}$ and $B=\wt D \ge 0$). On $\wt Y$, the orbifold tangent sheaf of the pair $(\widetilde X,\widetilde D)$, i.e. $\sT_{(\wt Y,f, \widetilde D)}$,  coincides with (the sheaf of sections of) the vector bundle defined as the locally free $\mathcal O_{\wt Y}$-module generated by $$z_1^{e_1}\frac{\d}{\d z_1}, \ldots, z_p^{e_p}\frac{\d}{\d z_p}, \frac{\d}{\d z_{p+1}}, \ldots, \frac{\d}{\d z_n}$$
whenever $B$ is locally given by $(z_1^{e_1}\cdots z_p^{e_p}=0)$ \--- up to relabelling the coefficients $e_i$. In order to lighten notation slightly, we will denote by $T_{\wt Y}(-\log B)$ the latter vector bundle and identify it freely with its associated (locally free) sheaf of sections.

\subsection{The semistability theorem}
We can now prove 
the slope-semistability of orbi-tangent sheaves for minimal pairs of log-general type, noting that 
this theorem holds irrespective of the choice of the orbifold structure, consisting of adapted or 
strictly adapted charts.

\begin{theorem}[Semistability of orbi-tangent sheaves]
\label{thm:MainSS}

In the setting \ref{setting}, assume that $K_X+D$ is nef and big. Then $\sT_{(\wt Y, f, \wt D)}$ is semistable with respect to $f^*\pi^*(K_X+D)$. 
\end{theorem}

\begin{remark}
In the course of the proof, we actually do not use the bigness assumption. However, semistability with respect to an arbitrary nef class is not a very meaningful notion. 
\end{remark}


\begin{proof}
Let $\sigma_i$ (resp. $t_j$) a section of $\mathcal O_X(\wt D_i)$ (resp. $\mathcal O_X(E_j))$ cutting out $\wt D_i$ (resp. $E_j$). One chooses smooth hermitian metrics $h_i$ (resp. $h_j$) on these bundles and a Kähler form $\omxt$ on $\wt X$. Finally, let $\om_X\in c_1(K_X+D)$ be a smooth form (not necessarily semipositive).
Because of \eqref{canonical}, there exists a smooth volume form $dV$ on $\wt X$ such that $-\Ric(dV)+\sum_i (1-\frac{b_i}{a_i}) \Theta_{h_i}(\wt D_i)= \pi^*\om_X+ \sum_j c_j \Theta_{h_j}(E_j)$. \\


\noindent
For $t>0$, the cohomology class $\{\pi^*\om_X+t\omxt\}$ is Kähler; therefore, one can solve for any $\ep>0$ the following Monge-Ampère equation: 
\begin{equation}
\label{MA}
(\pi^*\om_X+t\omxt+\ddc \vp_{t, \ep})^n= \frac{\prod(|t_j|^2+\ep^2)^{c_j(1-t)}e^{\vp_{t,\ep}}dV}{\prod |\sigma_i|^{2\left(1-\frac{b_i}{a_i}\right)}}
\end{equation}
and obtain a Kähler current $\omte:=\pi^*\om_X+t\omxt+\ddc \vp_{t, \ep}$ that is smooth outside $\Supp(\wt D)$ and has mixed conic and cusp singularities along $\wt D$ \--- with cone angles $\frac{2\pi b_i}{a_i}$ along $\wt D_i$ if $a_i<+\infty$, or cusp singularities along $\wt D_i$ otherwise. This follows from \cite[Thm 6.3]{GP}, as explained in \S \ref{mixed}.\\

\noindent
Moreover, $\omte$ is an approximation of the Kähler-Einstein metric in the sense that: 
\begin{equation}
\label{KE}
\Ric \omte = -\omte+t \omxt-((1-t)\Theta_{\ep}+t \Theta(E))+[\wt D]
\end{equation}
where $\Theta_{\ep}:=\sum_j c_j\left( \frac{\ep^2|D't_j|^2}{(|t_j|^2+\ep^2)^2}+\frac{\ep^2 \Theta_{h_j}(E_j)}{|t_j|^2+\ep^2}\right)$ is an approximation of $[E]$, the current of integration along $E$ and $\Theta(E):=\sum_j c_j \Theta_{h_j}(E_j)$. The reason we change $c_j$ into $(1-t)c_j$ is to prevent $\vp_{t,\ep}$ from getting unbounded along $\sum_{c_j=-1}E_j$ when $\ep\to 0$, $t>0$ being fixed. Pulling back $\omte$ by $f$, one gets a positive current $f^*\omte$ which satisfies: 

\vspace{2mm}
\begin{enumerate}
\item\label{item:PB1} $f^*\omte$ is smooth outside $\Supp(D'+H')=\Supp(B)$.
\vspace{2mm}
\item\label{item:PB2} $f^*\omte$ has cone angles $\frac{2\pi Nb_i}{a_i}$ along $D'_i$ (if $a_i<+\infty$)  and $2\pi N$ along $H'$; and $f^*\omte$ has cusp singularities along $\lceil D' \rceil$. Equivalently, $f^*\omte$ has cone angles $2\pi(1-e_i)$ along $B_i$ if $e_i<1$ and cusp singularities along $B_i$ if $e_i=1$.
\vspace{2mm}
\item\label{item:PB3} $\Ric(f^*\omte)=-f^*\omte+tf^*\omxt-(1-t)f^*\Theta_{\ep}-tf^*\Theta(E)-[B].$
\vspace{2mm}
\end{enumerate}
This is a consequence of the third property of the Kawamata cover $\gamma$ recalled above, cf last paragraph of \S \ref{mixed}. Let us recall what Item~\ref{item:PB2} means. Denoting $B_-$ (resp. $B_+$) the union of components of $B$ with negative coefficients (resp. positive coefficients, or equivalently coefficients equal to $1$) and choosing a chart $\Omega\subset \wt Y$ along with local coordinates $(z_1, \ldots, z_n)$ for which  $\mathrm{Supp}(B_-) \cap \Omega=(z_1\cdots z_p=0)$ and $B_+\cap \Omega =(z_{p+1} \cdots z_r =0)$, then $f^*\omega_{t,\ep}$ is quasi-isometric (on $\Omega$) to the model metric below:

\begin{equation}
\label{model}
\sum_{k=1}^p |z_k|^{-2e_k} idz_k \wedge d\bar z_k+\sum_{k=p+1}^r \frac{i dz_k \wedge d\bar z_k}{|z_k|^2 \log^2 \frac{1}{|z_k|^2}} +\sum_{k=r+1}^n i dz_k \wedge d\bar z_k
\end{equation}
Recall that for $1 \le k \le p$ one has $e_k\le 0$ so that the model metric above has \textit{zeros} along $B_-$ and "\textit{poles}" along $B_+$. \\

From now on, one will set $\om:=f^*\omte$. Because of~\ref{item:PB1}-\ref{item:PB2}, $\om$ induces a bounded hermitian metric on $T_{\wt Y}(-\log B)$ which is smooth outside $\Supp(B)$.
The strategy is to use this metric to derive the semistability property of the bundle; it is inspired from \cite{CP2} and \cite{GSS}. \\
 
If $\mathscr F$ is a reflexive subsheaf of $ T_{\wt Y}(-\log B)$ of rank $p$, then it induces a generically injective map of sheaves $(\Lambda^p \mathscr F)^{**}\longrightarrow \Lambda^p  T_{\wt Y}(-\log B)$. Setting $L:=(\Lambda^p \mathscr F)^{**}$ to be the determinant of $\mathscr F$, one gets a non-zero section $u$ of $\Lambda^p T_{\wt Y}(-\log B)\otimes L^{-1}$. As $\om$ is smooth (and Kähler) outside $\Supp(B)$, it induces a smooth hermitian metric $h$ on the vector bundle $T_{\wt Y}(-\log B)$ on that locus. Let us choose some fixed smooth metric $h_L$ on $L$. This enables one to compute the squared norm $|u|^2$ of the section $u$ with respect to $h\otimes h_L^{-1}$; this is a bounded function on $\wt Y$. In order to apply the vector bundle version of Lelong-Poincaré formula to $\ddc \log |u^2|_{h\otimes h_L^{-1}}$, one needs some additional regularization processes.

First, as $|u|^2$ may not be smooth along $\Supp (B)$, one introduces $(\chi_{\eta})_{\eta>0}$ a family of cut-off functions for $\Supp(B)$; one can arrange that the $L^1$ norm of $\ddc \chi_{\eta}$ (computed with respect a smooth metric on $\wt Y$) tends to zero when $\eta\to 0$, cf e.g. \cite[\S 9]{CGP}. Then to prevent $\log |u|^2$ from being unbounded (below), one chooses a constant $\lambda>0$, and evaluate the smooth quantity (outside $\Supp(B)$):

\begin{equation}
\label{eqt}
\ddc \log(\ud) = \frac{1}{\ud}\left( |D'u|^2-\frac{|\la D'u,u\ra |^2}{\ud}-\la \Theta_{h \otimes h_L^{-1}}(\Lambda^p T_{\wt Y}(-\log B) \otimes L^{-1} )u,u\ra \right) 
\end{equation}
Here, $D'$ is the $(1,0)$-part of the Chern connection of $h\otimes h_L^{-1}$ and for a (vector-valued) $(1,0)$-form $\alpha$, one uses the notation $|\alpha|^2:=\alpha \wedge \bar \alpha$. Outside the support of $B$, one can identify $T_{\wt Y}(-\log B)$ with $T_{\wt Y}$; moreover, one will identify in the following $\Theta_h(\Lambda^p T_{\wt Y})$ with $\Theta_h(\Lambda^p T_{\wt Y}) \otimes \mathrm{Id}_{L^{-1}}$. Finally, as $|\la D'u,u\ra |^2 \le |D'u|^2 \cdotp |u|^2$, one deduces from the equality \eqref{eqt} above: 
\begin{equation}
\label{ineq}
\ddc \log(\ud) \ge \frac{|u|^2}{\ud} \left( \Theta_{h_L}(L)-\frac{\la\Theta_h(\Lambda^p T_{\wt Y})u,u\ra}{|u|^2} \right) 
\end{equation}
Wedge inequality \eqref{ineq} with $ \chi_{\eta} \,\om^ {n-1}$ and integrate on $\wt Y$:\\
\begin{equation}
\label{slope}
 - \int_{\wt Y} \log(\ud) \, \ddc \ce \wedge \om^ {n-1} \ge \int_{\wt Y}  \frac{\ce|u|^2}{\ud} \left( \Theta_{h_L}(L)-\frac{\la\Theta_h(\Lambda^p T_{\wt Y})u,u\ra }{|u|^2} \right) \wedge \om^ {n-1}   
\end{equation}

\vspace{2mm}

Recall that $t,\ep$ being fixed, $\om$ is equivalent to the model metric given in \eqref{model}. In particular, it follows that $\pm \ddc \chi_{\eta} \wedge \om^{n-1}$ is uniformly dominated by the volume form of a metric with cusp singularities along $\Supp(B)$ \--- whose mass is finite. As $\chi_{\eta}$ converges smoothly to zero outside $\Supp(B)$, Lebesgue's dominated convergence theorem shows that the left hand side converges to $0$ when $\eta$ tends to $0$, $\lambda>0$ being fixed. By the same token, the first integral in the right hand side converges to $\int_{\wt Y} \Theta_{h_L}(L) \wedge \om^{n-1}$ when $\eta,\lambda$ approach zero. As the potentials of $\om$ have finite energy (cf \cite[Prop. 2.3]{G12}), this integral is nothing but the intersection number $L \cdotp \{\om\}^{n-1}=c_1(\mathscr F) \cdotp \left( f^*\pi^*(K_X+D)+t f^*\{\omxt\}\right)$.
So one is left to estimating the second integral in the right hand side. \\

By the symmetries of the curvature tensor, one has the following identity (outside $\Supp(B)$): $n\Theta_{\om}(T_{\wt Y}) \wedge \om^{n-1} =  (\sharp \Ric \om) \, \om^n$ where $\sharp:\Omega_{\wt Y}^{0,1}\to T^{1,0}_{\wt Y}$ is the standard isomorphism induced by $\om$ which extends to an operator $\Omega_{\wt Y}^{1,1}\to \mathrm{End}(T_{\wt Y})$. As a result, $$n\Theta_{h}(\Lambda^p T_{\wt Y})\wedge \om^{n-1}= (\sharp \Ric \om)^{\wedge p} \, \om^n$$ where for any endomorphism $f$ of an $n$-dimensional vector space $V$, one defines $f^{\wedge p}:\Lambda^pV\to \Lambda^pV$ by $f^{\wedge p}(v_1 \wedge \ldots \wedge v_p):=\sum_{k=1}^p v_{1}\wedge \ldots \wedge v_{k-1}\wedge f(v_k) \wedge v_{k+1}\wedge \ldots \wedge v_n$. These two operations preserve positivity (of $(1,1)$ forms and hermitian endomorphisms respectively). Moreover, it is easily checked that for any $(1,1)$-form $\alpha$, one has $\mathrm{tr}_{\rm End}(\sharp \alpha)  =\tr_{\om} \alpha$, and if $f$ is any positive semidefinite endomorphism of $V$, then $\tr_{\rm End}(f^{\wedge p}) \le {n \choose p} \,\tr_{\rm End}(f)$. \\

Thanks to item \ref{item:PB3} in the properties of $\om=f^*\omte$ one sees that outside $\Supp(B)$, $$(\sharp \Ric \om)^{\wedge p} = -p\mathrm{Id}_{\Lambda^pT_{\wt Y}}+t \sharp f^*(\omxt)^{\wedge p}-(1-t)(\sharp f^*\Theta_{\ep})^{\wedge p}-t \sharp f^*(\Theta(E))^{\wedge p}$$ 
Let $\omx$ be a reference Kähler form on $ \wt Y$; there exists $C>0$ such that $0\le f^*\omxt \le C \omx$, and therefore 
\begin{equation}
\label{maj}
\tr_{\rm End} (f^*\omxt)^{\wedge p} \, \om^n \le n C {n \choose p}\,  \omx \wedge \om^{n-1}
\end{equation}
As a consequence, $\int_{\wt Y} \frac{\la t (\sharp f^*\omxt)^{\wedge p}u,u\ra}{|u|^2}\cdotp \om^ {n}  \le t C' \int_{\wt Y}\omx \wedge \om^{n-1}$ which is independent of $\ep$ and tends to zero when $t$ goes to zero. The same argument shows that $\int_{\wt Y} \frac{\la t (\sharp f^*\Theta(E))^{\wedge p}u,u\ra}{|u|^2}\cdotp \om^ {n}$ converges to zero when $\ep$ and then $t$ approach zero. \\

Furthermore, $f$ is generically unramified along $\Supp(E)$ hence $f^*E$ has simple normal crossings support and $f^*\Theta_{\ep}$ is a standard approximation of the current of integration along $f^*E$. Write $f^*E=\sum c_j E_j'$; then $f^*\Theta_{\ep}=\sum c_j (\alpha_{j,\ep}+\beta_{j,\ep})$ with $ \alpha_{j,\ep}=\frac{\ep^2}{(|t_j'|^2+\ep^2)^2}\cdotp |D't_j'|^2$ and $\beta_{j,\ep}=\frac{\ep^2 }{|t_j'|^2+\ep^2} \cdotp f^*\Theta_{h_j}(E_j)$ where $t_j'=f^*t_j$. Until the end of this paragraph, one will drop the indexes $j$ and $\ep$ to lighten notation. One needs to evaluate the quantity $(\sharp(\alpha+\beta))^{\wedge p} \om^n$. The term $(\sharp \alpha)^{\wedge p} \om^n$ is positive and dominated by $n {n \choose p} \mathrm{Id}_{\Lambda^pT_{\wt Y}} \cdotp \alpha \wedge \om^{n-1}$. The term $(\sharp \beta)^{\wedge p} \om^n$ is dominated (in norm) by $\frac{C \ep^2 }{|t'|^2+\ep^2}\cdotp \mathrm{Id}_{\Lambda^pT_{\wt Y}} \cdotp f^*\omxt \wedge \om^{n-1}$. Therefore,
\begin{eqnarray*}
\left|\frac{\la\sharp(\alpha+\beta))^{\wedge p}u,u\ra}{|u|^2} \right| \om^n & \le & C \alpha \wedge \om^{n-1}+\frac{C \ep^2 }{|t'|^2+\ep^2}\cdotp f^*\omxt \wedge \om^{n-1}\\
& = &  C (\alpha+\beta) \wedge \om^{n-1}-C \beta \wedge \om^{n-1}+\frac{C \ep^2 }{|t'|^2+\ep^2}\cdotp  f^*\omxt \wedge \om^{n-1}\\
& \le & C (\alpha+\beta) \wedge \om^{n-1}+\frac{C' \ep^2 }{|t'|^2+\ep^2}\cdotp  f^*\omxt \wedge \om^{n-1}
\end{eqnarray*}
so that: 
$$\int_{\wt Y}\left|\frac{\la(\sharp f^*\Theta_{\ep})^{\wedge p}u,u\ra}{|u|^2}\wedge \om^{n-1}\right|\le C \int_{\wt Y}f^*\Theta_{\ep}\wedge \om^{n-1}+C\sum_j \int_{\wt Y}\frac{ \ep^2 }{|t'_j|^2+\ep^2}\cdotp  f^*\omxt \wedge \om^{n-1}$$
As $E$ is $\pi$-exceptional, the first integral on the right hand side equals $t E \cdotp \{\omxt\}$; in particular, it does not depend on $\ep$ and converges to $0$ when $t\to 0$. The term equals $C |G| \sum_j \int_{\wt X} \frac{ \ep^2 }{|t_j|^2+\ep^2}\cdotp\omxt\wedge \omte^{n-1}$ which converges to zero by Lemma \ref{int:zero} stated at the end of this proof. \\

Putting everything together, one obtains: \\
$$ \int_{\wt Y}  \frac{\ce|u|^2}{\ud}\frac{\la\Theta_h(\Lambda^p T_{\wt Y})u,u\ra }{|u|^2}  \wedge \om^ {n-1}   = I_1+I_2+I_3$$
where: $$I_1=-\frac p n \int_{\wt Y}\frac{\ce|u|^2}{\ud} \om^n, \quad I_2=\frac t n \int_{\wt Y}\frac{\ce|u|^2}{\ud}\frac{\la( \sharp f^*(\omxt-\Theta(E)))^{\wedge p}u,u\ra}{|u|^2}\wedge \om^{n-1}$$ and $$I_3=-\frac {1-t} n \int_{\wt Y}\frac{\ce|u|^2}{\ud}\frac{\la(\sharp f^*\Theta_{\ep})^{\wedge p}u,u\ra}{|u|^2}\wedge \om^{n-1}$$ 

\noindent
We claim that $I_2$ and $I_3$ converge to zero when $\ep, t$ go to $0$, say for $\eta=\lambda=0$. More precisely, one can dominate point-wise the absolute value of the integrands of $I_2$ and $I_3$ by an $(n,n)$-form $\alpha_{t,\ep}$ independent of both $\eta>0$ and $\lambda>0$ that satisfies $\lim_{t\to 0}\lim_{\ep \to 0} \int_{\widetilde Y}\alpha_{t,\ep}=0$. Dominated convergence enables to first pass to the limit when, in that order: $\eta, \lambda, \ep, t$ converge to zero.

Moreover, dominated convergence show that $I_1$ converges to  $-\frac p n\int_{\wt Y}\om^n$ when $\eta, \lambda$ converge to $0$. As the potentials of $\om$ have finite energy, this integral is nothing but $-\frac p n \{\om\}^n$ and this quantity (depending on $t$ only) converges to $-\frac p n ([f^*\pi^*(K_X+D)]^n)$ when $t\to 0$. Combining this with \eqref{slope}, one finds:
$$n  \, c_1( \mathscr F) \cdotp [f^*\pi^*(K_X+D)]^{n-1} \le -p \,  ([f^*\pi^*(K_X+D)]^n)$$
Now the determinant of the dual of $T_{\wt Y}(-\log B)$ is $K_{\wt Y}+B=f^*\pi^*(K_D+D)+f^*E$. As $E$ is $\pi$-exceptional, the intersection of $c_1(T_{\wt Y}(-\log B))$ with  $[f^*\pi^*(K_X+D)]^{n-1}$ is precisely $-([f^*\pi^*(K_X+D)]^n)$, hence semistability of $T_{\wt Y}(-\log B)$ follows. 
\end{proof}

\vspace{3mm}

\noindent
In the course of the proof above, we used the following result:

\begin{lemma}
\label{int:zero}
The positive current $\om_{t, \ep}$ solution of the Monge-Ampère equation \eqref{MA} satisfies, for each fixed $t>0$ and each index $j$:
$$\lim_{\ep \to 0} \int_{\wt X} \frac{\ep^2}{|t_j|^2+\ep^2} \omxt \wedge \om_{t,\ep}^{n-1}=0$$ 
\end{lemma}

\begin{remark}
\label{remark}
Note that because the potentials of $\om_{t, \ep}$ are unbounded, one cannot directly use an argument based on Chern-Levine-Nirenberg inequality as in \cite[Claim 9.5]{GGK}. Moreover, because of the conic part in \eqref{MA}, one cannot reproduce the arguments of \cite[Lem. 2.1]{GSS} based on \cite[Prop. 1.1]{GSS} as it would involve (among other things) having at hand a model conic/cusp metric with bisectional curvature bounded above and, as far as we know, such a metric has not yet been proved to exist. 
\end{remark}

\begin{proof}[Proof of Lemma~\ref{int:zero}]
Let $V_{\ep}:=\{|t_j|^2<\ep\}$. One has 
$$ \int_{\wt X \smallsetminus V_{\ep}} \frac{\ep^2}{|t_j|^2+\ep^2} \omxt \wedge \om_{t,\ep}^{n-1} \le \ep \int_{\wt X}\omxt \wedge \om_{t,\ep}^{n-1}$$ 
and this quantity is dominated by $\ep \{\omxt\} \cdotp \{\om_{t,\ep}\}^{n-1}$ which obviously converges to zero. Here we have used that because the potentials of $\omte$ have finite energy, the mass of the (mixed) Monge-Ampère products of $\omte$ is computed in cohomology. We are left to prove the following:
\begin{equation}
\label{inte}
\lim_{\ep \to 0} \int_{V_{\ep}} \omxt \wedge \om_{t,\ep}^{n-1}=0
\end{equation}
The first observation is that thanks to \cite[Thm. A]{GW}, the potentials $\vp_{t,\ep}$ of $\omte$ satisfy 
\begin{equation}
\label{estimate}
\vp_{t,\ep}= \sum_{a_i=+\infty}-\log \log^2 |\sigma_i|^2 +R_{t,\ep}
\end{equation}
where $\sup_{\wt X}|R_{t,\ep}| \le C_t$ for some constant $C_t>0$ independent of $\ep$. 
Moreover, one has uniform $\mathcal C^{k}$ estimates for $\vp_{t,\ep}$ on any compact subset of $\wt X \smallsetminus \supp(D+E)$; therefore any weak limit $\psi_t$ of $\vp_{t,\ep}$ (when $\ep$ approaches $0$) satisfies: 
\begin{enumerate}
\item[$\bullet$] $\psi_t$ is smooth on $\wt X \smallsetminus \supp(D+E)$, and satisfies the following equation on that locus:
\begin{equation}
\label{MA2}
(\pi^*\om_X+t\omxt+\ddc \psi_{t})^n= \frac{\prod |t_j|^{2c_j(1-t)}e^{\psi_{t}}dV}{\prod |\sigma_i|^{2\left(1-\frac{b_i}{a_i}\right)}}
\end{equation}
\item[$\bullet$] $\sup_{\wt X} \left|\psi_t - \sum_{a_i=+\infty}-\log \log^2|\sigma_i|^2\right|<+\infty$
\end{enumerate}
It follows from the second point and \cite[Prop. 2.3]{G12} that $\psi_t$ has finite energy with respect to any Kähler form $\om$ such that $\om+\ddc \psi_t \ge 0$. In particular, the equation \eqref{MA2} is satisfied on the whole $\wt X$ (as the Monge-Ampère of $\psi_t$ puts no mass on pluripolar sets). As \eqref{MA2} admits a unique solution (by comparison principle, cf \cite[Prop. 4.1]{BG}), all subsequential limits of $(\vp_{t, \ep})_{\ep>0}$ when $\ep\to 0$ agree, and therefore $\vp_{t,\ep}$ converges weakly to the solution $\vp_t$ of \eqref{MA2}.
We want to show that the convergence is actually strong. For that purpose, one observes that the quantity 
$$I(\vp_t, \vp_{t,\ep}):=\int_{\wt X} (\vp_t- \vp_{t,\ep})(\mathrm{MA}(\vp_{t,\ep})-\mathrm{MA}(\vp_t))$$ 
converges to zero when $\ep\to 0$ thanks to Lebesgue dominated convergence theorem (by \eqref{estimate}, $\vp_t- \vp_{t,\ep}$ is uniformly bounded). By \cite[Prop. 2.3]{BBEGZ}, this implies that $\vp_{t,\ep}$ converges strongly to $\vp_t$. In particular, the measures $(\omxt+\om_{\vp_{t,\ep}})^n= (\omxt+ \pi^*\om_X+t\omxt+\ddc \vp_{\ep,t})^n$ converge weakly to the non-pluripolar measure $(\omxt+\om_{\vp_{t}})^n$ when $\ep \to 0$, $t>0$ being fixed.\\

\noindent
Now, assume that \eqref{inte} does not hold. As $(\omxt+\om_{\vp_{t, \ep}})^n\ge \omxt \wedge \om_{\vp_{t, \ep}}^{n-1}$, this would imply that $$\limsup_{\ep \to 0} \int_{V_{\ep}} (\omxt+\om_{\vp_{t, \ep}})^n>0$$
In particular, denoting $M_t:=\{\omxt+ \pi^*\om_X+t\omxt\}^n$, one could find $\delta>0$ and a sequence $(\ep_k)_{k\ge 0}$ converging to $0$ such that for any $k\ge 0$:
$$\int_{\wt X\smallsetminus V_{\ep_k}} (\omxt+\om_{\vp_{t, \ep_k}})^n \le M_t - \delta$$
In particular, given $k_0\ge 0$, one would have for any $k \ge k_0$:
$$\int_{\wt X\smallsetminus V_{\ep_{k_0}}} (\omxt+\om_{\vp_{t, \ep_k}})^n \le M_t - \delta$$
As weak convergence of measures does not increase the mass, one would find:
$$\int_{\wt X\smallsetminus V_{\ep_{k_0}}} (\omxt+\om_{\vp_{t}})^n \le M_t - \delta$$
As this holds for any $k_0\ge 0$, one deduces:
$$\int_{\wt X\smallsetminus \{t_j=0\}} (\omxt+\om_{\vp_{t}})^n \le M_t - \delta$$
which contradicts the fact that $\vp_t$ has finite energy with respect to $(1+t)\omxt+\pi^*\om_X$. Therefore \eqref{inte} is proved, and the lemma follows. 
\end{proof}

%
%

 \subsection{Proof of Theorem~C}
 Let $f: Y\to (X, D= \sum d_i\cdot D_i)$ be a strictly adapted morphism, 
 provided for instance by Example~\ref{ex:QFactorial}. We note that $Y$ is not 
 necessarily lc. Let  $\pi: \widetilde X \to X$ be a strong log-resolution for the 
 pair $(X, D)$. Set $\widetilde D\subset \widetilde X$ to be the birational transform of 
 $D$ by the morphism $\pi$. Let $\widehat Y$ be the irreducible component of the normalization of 
 $(Y\times_{X} \widetilde X)$ inducing surjective morphisms $\widehat \pi: \widehat Y \to Y$, 
 $\widehat f: \widehat Y \to \widetilde X$ and the corresponding commutative diagram:

 $$
  \xymatrix{
    \widehat Y \ar[rrr]^{\widehat f}  \ar[d]_{\widehat \pi} &&& \widetilde X  \ar[d]^{\pi} & &&  \widetilde Y \ar[lll]_{\widetilde f}\\ 
      Y \ar[rrr]^{f}       &&& X &&&
  }
  $$
 Note that as $\pi$ is a projective birational morphism, the finite morphism 
 $\widehat f: \widehat Y \to \widetilde X$ is Galois. Let $
 \widetilde f: \widetilde Y\to (\widetilde X,\widetilde D)$
 be an adapted morphism such that the two orbi-cotangent sheaves 
 $\Omega^{1}_{(\widetilde{Y}, \widetilde f, \widetilde D)}$ and 
 $\Omega^{1}_{(\widehat{Y}, \widehat f ,\widetilde D)}$
 are compatible (see Example~\ref{ex:compatible}).

 Now, according to Theorem~\ref{thm:MainSS} we know that  
 $\Omega^{1}_{(\widetilde{Y},\widetilde f,  \widetilde D)}$ is semistable 
 with respect to $\pi^*(K_X+D)$. Therefore, by Proposition~\ref{prop:IndepSS}, 
 so is $\Omega^{1}_{(\widehat{Y}, \widehat f , \widetilde D )}$. We now claim that 
 $\Omega^{1}_{(Y, f, D)}$ is semistable with respect to $K_X+D$. 
 
 \noindent
 Otherwise, there exists a reflexive subsheaf $\F\subset \Omega^{1}_{(Y, f, D)}$
 such that 
 $$
 \mu_{f^*(K_X+D) }(\F)  >  \mu_{f^*(K_X+D)}(\Omega^{1}_{(Y, f ,D)}).  
 $$
 But then the subsheaf $\widehat \F \subset \Omega^{1}_{(\widehat Y, \widehat f, D)}$
 defined by 
 $$
 \widehat \F  : = (\widehat \pi^*\F) \cap \Omega_{(\widehat Y, \widehat f, D)}^{1},
 $$
 which verifies the equality $\mu_{\widehat \pi^*f^*(K_X+D)} (\widehat \F)= 
 \mu_{f^*(K_X+D)}(\F)$, destabilizes $\Omega^{1}_{(\widehat Y, \widehat f, \widetilde D)}$;
 a contradiction. 
 
 \begin{remark}
 We would like to emphasize that, as evident from the proof, Theorem~C is valid for \emph{any}
 choice of an orbifold structure for $(X,D)$ and not necessarily those that are strict. 
 \end{remark}

\section{Orbifold Miyaoka-Yau inequality for minimal pairs}
\label{section5-MYInequality}

 \subsection{Proof of Theorem~\ref{thm:main-ineq}} 
 Let $H$ be the ample divisor in 
 Proposition~\ref{prop:OrbiLc} and set 
 $\mathcal C_{\underline \alpha(m)}= \{ (U_{\underline \alpha(m)}, f_{\underline \alpha(m)} , X_{\underline \alpha(m)}) \}$ 
 to be the orbi-structure for the klt pair $(X, D_m)$ in codimension two.
 To avoid unnecessarily cumbersome notations we will assume that
 $D- \lfloor D \rfloor =0$.

   \vspace{2 mm}
 \noindent
\textbf{Step 1.} \emph{Restriction to a general surface $S_m$.}
 
 \noindent
 
 Let $f'_m: Y_m \to (X, D_m)$ be a morphism strictly adapted to $D_m$ decomposed 
 as in (\ref{decomp}). 
 By Theorem~\ref{thm:SS}, the sheaf $\Omega^{1}_{(Y_m, f'_m, D_m)}$ is semistable with respect to 
 $K_X+D_m$. As semistability is determined in codimension one, its coherent extension $\Omega^{[1]}_{(Y_m, f'_m, D_m)}$
 is also semistable, the latter being a reflexive sheaf on $Y_m$.
 Now, according to Flenner's restriction theorem (\cite{Flenner84}) ---which generalizes
 Mehta-Ramanathan's theorem to normal varieties---
 for sufficiently large, positive 
 integer $a_m$, there is a complete intersection curve 
 $C_m:=D_1\cap \ldots \cap D_{n-1}$, 
 where $D_i$ are general members of $| a_m\cdot (K_X+D_m) |$ such that 
 $( \Omega^{1}_{(Y_m , f'_m, D_m)}|_{C_{Y_m}} )$ is semistable, where $C_{Y_m}: = (f'_m)^{-1}(C_m)$.

 We observe that as $\Omega^1_{(Y_m, f'_m, D_m)}$ and $\Omega^1_{\mathcal C_{\underline \alpha(m)}}$ are compatible
 in codimension one, 
 thus so are their restrictions $\Omega^1_{(Y_m, f_m, D_m)}|_{C_{Y_m}}$ and $\Omega^1_{\mathcal C_{\underline \alpha(m)}}|_{C_m}$
 \footnote{By the notation $\Omega^1_{\mathcal C_{\underline \alpha(m)}}|_{C_m}$ we mean
 the pullback of $\Omega^1_{\mathcal C_{\underline \alpha(m)}}$ to the orbi-sheaf on $C_m$ with its naturally induced 
 orbifold structure. 
 We refer to \cite[Subsect.~3.6.2]{GKPT15} for more details in the more classical setting
 of $\mathcal Q$-varieties (which readily generalizes to the current context).}.
 Therefore, using Proposition~\ref{prop:IndepSS}, it follows that $\Omega_{\mathcal C_{\underline \alpha(m)}}^{1}|_{C_m}$
 is also semistable.
 In particular,  for $S_m:=D_1\cap \ldots \cap D_{n-2}$, the restriction $\Omega_{\mathcal C_{\underline \alpha(m)}}^{1}|_{S_m}$
 is semistable with respect to the orbifold structure $\mathcal C_{S_m,{\underline \alpha(m)}}$ on $S_m$ naturally induced by 
 the restriction of $\mathcal C_{\underline \alpha(m)}$.\\

  \vspace{2 mm}
 \noindent
 \textbf{Step 2.} \emph{Construction of a stable orbi-Higgs sheaf.}
 
 \noindent
 Let $\O_{\mathcal C_{\underline \alpha(m)}}$ is the orbi-sheaf defined by $\{\O_{X_{\underline \alpha(m)}} \}$.
 Define the orbi-Higgs sheaf by $\E_{\mathcal C_{\underline \alpha(m)}}:= \Omega^{1}_{\mathcal C_{\underline \alpha(m)}} \oplus \O_{\mathcal C_{\underline \alpha(m)}}$
 together with the canonically defined Higgs map $ \theta_{\mathcal C_{\underline \alpha(m)}}: \E_{\mathcal C_{\underline \alpha(m)}} 
 \to \E_{\mathcal C_{\underline \alpha(m)}}\otimes \Omega^1_{\mathcal C_{\underline \alpha(m)}}$ defined by the isomorphism 
 $$
\Omega_{\mathcal C_{\underline \alpha(m)}}^{1} \to \O_{\mathcal C_{\underline \alpha(m)}} \otimes \Omega_{\mathcal C_{\underline \alpha(m)}}^{1}
 $$
 when restricted to the first factor and zero when restricted to $\O_{\mathcal C_{\underline \alpha(m)}}$.

 Now, consider the restriction 
 $$\E_{\mathcal C_{S_m, \underline \alpha_m}}:=\E_{\mathcal C_{m}}|_{S_m}= 
 \Omega_{\mathcal C_{\underline \alpha(m)}}^{1}|_{S_m} \oplus \, \mathcal O_{\mathcal C_{\underline \alpha(m)}}|_{S_m}=
 \Omega_{\mathcal C_{\underline \alpha(m)}}^{1}|_{S_m}  \oplus \, \mathcal O_{\mathcal C_{S_m, \underline \alpha_m}},
 $$
 with the Higgs field $\theta_{\mathcal C_{S_m, \underline \alpha_m}}$ defined by the composition
 \begin{equation*}
 \xymatrix{
  \E_{\mathcal C_{S_m, \underline \alpha_m}} \ar[rr]^(.4){ \theta_{\mathcal C_{\underline \alpha(m)}}|_{S_m}}  &&  \E_{\mathcal C_{S_m, \underline \alpha_m}} \otimes \Omega^1_{\mathcal C_{\underline \alpha(m)}}|_{S_m}
    \ar[r]  &   \E_{\mathcal C_{S_m, \underline \alpha_m}} \otimes \Omega_{\mathcal C_{S_m, \underline \alpha_m}}^1.
 }
 \end{equation*}
As $S_m$ is general and $\E_{\mathcal C_{\underline \alpha(m)}}$ is reflexive, it follows that $\E_{\mathcal C_{S_m, \underline \alpha_m}}$ is locally free. 
(To be clear, $\E_{\mathcal C_{S_m, \underline \alpha_m}}$ is \textit{not} the 
orbi-sheaf $\Omega^{1}_{\mathcal C_{S_m, \underline \alpha_m}} \oplus \O_{\mathcal C_{S_m, \underline \alpha_m}}$.)
 \begin{claim}\label{claim:HiggsStable}
 The locally free orbi-Higgs sheaf $(\E_{\mathcal C_{S_m, \underline \alpha_m}}, \theta_{\mathcal C_{S_m, \underline \alpha_m}})$ is 
 Higgs stable with respect to $(K_X + D_m)|_{S_m}$.
 \end{claim}

\begin{proof}[Proof of Claim~\ref{claim:HiggsStable}]
 Let $\F_{\mathcal C_{S_m, \underline \alpha_m}} \subset \E_{\mathcal C_{S_m, \underline \alpha_m}}$ be an orbi-Higgs subsheaf. 
 As there are no non-trivial subsheaves of $(\Omega_{\mathcal C_{m}}^{1})|_{S_m}$
 that are invariant under the Higgs operator $\theta_{\mathcal C_{S_m, \underline \alpha_m}}$, the orbi-sheaf 
 $\F_{\mathcal C_{S_m, \underline \alpha_m}}$ must have a non-trivial projection onto an  orbi-subsheaf 
 $\sL_{\mathcal C_{S_m, \underline \alpha_m}}$ of $\O_{\mathcal C_{S_m, \underline \alpha_m}}$. Let 
 \begin{equation}\label{eq:HiggsExact}
 \xymatrix{
    0 \ar[r]  & \G_{\mathcal C_{S_m, \underline \alpha_m}}  \ar[r]  & \F_{\mathcal C_{S_m, \underline \alpha_m}}  \ar[r] & \sL_{\mathcal C_{S_m, \underline \alpha_m}} \
    \ar[r] & 0
  }
 \end{equation}
 be the resulting exact sequence of orbi-sheaves, where $\G_{\mathcal C_{S_m, \underline \alpha_m}}$ is a 
 locally free subsheaf of ${\Omega_{\mathcal C_{\underline \alpha_m}}^{1}}|_{S_m}$.
 The rest of our arguments are now very similar to those in~\cite[\S 7]{GKPT15}. 
 From~\eqref{eq:HiggsExact} and the inclusion $\sL_{\mathcal C_{S_m, \underline \alpha_m}} \subset \O_{\mathcal C_{S_m, \underline \alpha_m}}$
 it follows that 
 \begin{equation}\label{eq:SmallerChern}
  c_1(\F_{\mathcal C_{S_m, \underline \alpha_m}})  \leq c_1(\G_{\mathcal C_{S_m, \underline \alpha_m}}).
 \end{equation}
 Let $r : = \rank(\F_{\mathcal C_{\underline \alpha(m)}})$. By dividing the two sides of~\eqref{eq:SmallerChern} by $(r-1)$ and using the semistability of $\Omega^{1}_{\mathcal C_{\underline \alpha m}}|_{S_m}$, we find that 
 
 \begin{align*}
 \mu(\F_{\mathcal C_{S_m, \underline \alpha_m}}) &= \mu(\G_{\mathcal C_{S_m, \underline \alpha_m}}) \cdot \frac{r-1}{r}\\
        & \leq  \frac{1}{n} ( (K_X+D_m)|_{S_m} )^2 \cdot \frac{r-1}{r} 
                         \;\;  \;\;\;  \text{by the semistability of $\Omega^1_{\mathcal C_{\underline \alpha_m}}|_{S_m}$}  \\
        & <  \frac{1}{n+1} ( (K_X+D_m)|_{S_m} )^2    \\
        &= \mu \big(  \Omega^{1}_{\mathcal C_{\underline \alpha_m}}|_{S_m} \oplus \O_{\mathcal C_{S_m, \underline \alpha_m}}  \bigr),
  \end{align*}
 as required (here $\mu$ is the slope with respect to $K_X+D_m$).
 \end{proof}
 
 \vspace{2 mm}
 \noindent 
 \textbf{Step 3. }\emph{The Miyaoka-Yau inequality.} 
 
 \noindent
 Let $g_m: \widehat S_m \to S_m$ be the global Mumford cover associated to the orbi-structure 
 $\mathcal C_{S_m, \underline \alpha_m}$ on $S_m$ and $G_m:=\Gal (\widehat S_m/ S_m)$. 
Following Subsection~\ref{subsect:GlobalCover}, let $\widehat \E_{\mathcal C_{S_m, \underline \alpha_m}}$ be the locally free $G_m$-sheaf on $\widehat S_m$ associated with $\E_{\mathcal C_{S_m, \underline \alpha_m}}$. 
 It comes naturally equipped with a Higgs field $ \widehat \theta_{\mathcal C_{S_m, \underline \alpha_m}}$ such that  
 $$
 \widehat \theta_{\mathcal C_{S_m, \underline \alpha_m}}: \widehat \E_{\mathcal C_{S_m, \underline \alpha_m}} \to \widehat \E_{\mathcal C_{S_m, \underline \alpha_m}}\otimes \sW_m,
 $$
 where $\sW_m \subset \Omega^{[1]}_{\widehat S_m}$ is a locally free subsheaf.  
 Set $H_{m}:= (K_X+D_m)|_{S_m}$ and $\widehat H_{m}:=(g_m)^*H_{m}$.

 \begin{claim}\label{claim:GStable}
 For an equivariant resolution $\pi_m : \widetilde S_m \to \widehat S_m$, there exists an ample divisor 
 $\widetilde H_m\subset \widetilde S_m$ such that
 the locally free Higgs sheaf 
 $$
 (\widetilde \E_{\mathcal C_{S_m, \underline \alpha_m}}, \widetilde \theta_{\mathcal C_{S_m, \underline \alpha_m}}): = (\pi_m)^*(\widehat \E_{\mathcal C_{S_m, \underline \alpha_m}}
 , \widehat \theta_{\mathcal C_{S_m, \underline \alpha_m}})
 $$
is \emph{Higgs $G_m$-stable} with respect to $\widetilde H_m$, that is for the Higgs, $G_m$-subsheaf 
 $\widetilde \F_{\mathcal C_{S_m, \underline \alpha_m}}\subset \widetilde \E_{\mathcal C_{S_m, \underline \alpha_m}}$, 
 we have $\mu_{\widetilde H_m}(\widetilde \F_{\mathcal C_{S_m, \underline \alpha_m}})<
 \mu_{\widetilde H_m}(\widetilde \E_{\mathcal C_{S_m, \underline \alpha_m}})$. Moreover, one can 
 arrange that $(\pi_{m})_*\widetilde H_m=\widehat H_m$, as $1$-cycles on $\widehat S_m$.  
 \end{claim}
 
\begin{proof}[Proof of Claim~\ref{claim:GStable}] 
  First, we notice that from the orbi-Higgs stability of $(\E_{\mathcal C_{S_m, \underline \alpha_m}}, \theta_{\mathcal C_{S_m, \underline \alpha_m}})$ 
  it follows that the 
  locally free Higgs sheaf $(\widehat \E_{\mathcal C_{S_m, \underline \alpha_m}},\widehat \theta_{\mathcal C_{S_m, \underline \alpha_m}})$ 
  is $G_m$-stable, i.e. 
  \begin{equation}\label{eq:GStable}
  \mu_{\widehat H_m} (\widehat \F_{\mathcal C_{S_m, \underline \alpha_m}}) < \mu_{\widehat H_m}(\widehat \E_{\mathcal C_{S_m, \underline \alpha_m}})
  \end{equation}
  for every Higgs, $G_m$-subsheaf $\widehat \F_{\mathcal C_{S_m, \underline \alpha_m}} \subset \widehat \E_{\mathcal C_{S_m, \underline \alpha_m}}$. 
  This is because, every such subsheaf (which we may assume to be saturated) descends to an orbi-Higgs subsheaf
  $(\F_{\mathcal C_{S_m, \underline \alpha_m}}, \theta_{\mathcal C_{S_m, \underline \alpha_m}}) \subset (\E_{\mathcal C_{S_m, \underline \alpha_m}}, \theta_{\mathcal C_{S_m, \underline \alpha_m}})$ on $S_m$, 
  i.e. $\widehat \F_{\mathcal C_{S_m, \underline \alpha_m}} = q_{\alpha_m}^{*}(\F_{\mathcal C_{S_m, \underline \alpha_m}})$, cf.~\cite[Thm.~4.2.15]{HL10}.
  Here, the morphism $q_{\underline\alpha_m}: \widehat S_{\underline\alpha_m}\to S_{\alpha_m}$ 
  factors $g_m$: 
  
  $$
 \begin{tikzcd}
 \widehat S_{\mathcal C_{S_m, \underline \alpha_m}} \arrow{rrrr}{g_m}
 &&&& S_m \\
 \widehat S_{\underline\alpha_m} \arrow[hookrightarrow]{u} \ar{rr}{q_{\underline \alpha_m}}
 && S_{\underline \alpha_m} \ar{rr}{p_{\underline \alpha_m}}
 &&  V_{\underline \alpha_m}  \arrow[hookrightarrow]{u},
 \end{tikzcd}
 $$
 as defined in Subsection~\ref{subsect:GlobalCover}. Therefore the orbi-stability 
 of $(\E_{\mathcal C_{S_m, \underline \alpha_m}}, \theta_{\mathcal C_{S_m, \underline \alpha_m}})$ ensures that Inequality~\ref{eq:GStable} holds. 
 As a result, the pull-back Higgs bundle $(\widetilde \E_{\mathcal C_{S_m, \underline \alpha_m}}, \widetilde\theta_{\mathcal C_{S_m, \underline \alpha_m}})$ 
 is $G_m$-stable with respect to $\pi_m^*\widehat H_m$.
   
 Now, let $E$ be an effective exceptional divisor that is relatively anti-ample. As stability 
 is an open condition, it follows that for any sufficiently small $\epsilon \in \Q^+$ we can 
 guarantee that $\widetilde H_m: =( \pi_m^*\widehat H_m - \epsilon \cdot E)$ is ample and that 
 $(\widetilde \E_{\mathcal C_{S_m, \underline \alpha_m}}, \widetilde \theta_{\mathcal C_{S_m, \underline \alpha_m}})$ is Higgs stable 
 with respect $\widetilde H_m$. This finishes the proof of Claim~\ref{claim:GStable}.
 \end{proof}
 
 From the original result of Simpson on the existence of HYM metrics, cf.~\cite[Prop.~3.4]{Simpson}, 
 it thus follows that $\widetilde \E_{\mathcal C_{S_m, \underline \alpha_m}}$ verifies the Bogomolov-Gieseker inequality. 
 Hence, by projection formula, so does $\widehat \E_{\mathcal C_{S_m, \underline \alpha_m}}$  
 and thus the inequality 
 \begin{equation}\label{eq:Last}
 \bigl( 2(n+1)\cdot c_2(\Omega_{(\mathcal C_{\underline \alpha(m)}, D_m)}^{1}) -
            n \cdot c_1^2(\Omega_{(\mathcal C_{\underline \alpha(m)}, D_m)}^{1} ) \bigr)
                \cdot (K_X+ D+ \frac{1}{m}\cdot H)^{n-2}\geq 0
 \end{equation}
 holds for all $m\geq 2$.
 The inequality in Theorem~\ref{thm:main-ineq} now immediately follows from (\ref{eq:Last})
 by taking the limit $m\to \infty$, using the continuity property in Proposition~\ref{prop:continuity}. 
 
  \

\subsection{About the assumptions on the singularities}
\label{ssect:extension}
\

A more general setting to prove Miyaoka-Yau inequality would be the one for pairs $(X,D)$ such that  $(X,D)$ has log canonical singularities with $K_X+D$ nef. 
As mentioned in the Subsection~\ref{examples}, this context is too general for a workable definition of an orbifold second Chern class $c_2(X,D)$, even if the space $X$ is assumed to be klt. 

Looking carefully at the proof of Theorem~\ref{thm:main-ineq}, one may observe that what we really need is the existence of an orbi-étale structure in codimension two for $(X,D)$ as well as for $(X,D_m)$, cf. Proposition~\ref{prop:OrbiLc}. Both of these conditions are satisfied in two particular cases: if  $(X,D)$ is log smooth or if $X$ klt and $D$ is reduced.
\begin{enumerate}
\item[$\ast$] Assume that $(X,D)$ is log smooth. Then the existence of an orbi-étale structure on the whole $X$ follows from Example~\ref{ex:OrbiSmooth}. In particular, in view of Example~\ref{loglisse}, one has for any log smooth, log canonical pair $(X,D)$ such that $K_X+D$ is nef and big, say:
 $$    \bigl( 2(n+1)\cdot c_2(X,D) - n\cdot c_1^2(X,D) \bigr) \cdot (K_X+D)^{n-2} \geq 0$$
 where $c_2(X,D)=c_2(\Omega_X)+ c_1(\Omega_X)\cdotp D+\sum_{i<j}\big(1-\frac{b_i}{a_i}\big)\big(1-\frac{b_j}{a_j}\big)D_i\cdotp D_j + \sum_i\big(1-\frac{b_i}{a_i}\big)D_i^2$. This generalizes \cite[Prop.~4.2 and 4.5]{SW}.
\item[$\ast \ast$] Assume that $X$ is klt and that $D$ is reduced. The main point is to obtain \cite[Prop. 9.14]{GKKP} in this setting. This can be checked using the classification of pairs $(S,C)$ where $S$ is a klt surface and $C$ a reduced curve with $(S,C)$ lc, cf e.g. \cite[Chapt.~10 p.117]{Kollar92} combined with the reduction argument given in \cite[Sect.~9]{GKKP}.
\end{enumerate}


\section{Appendix. Products of orbifold Chern classes}\label{append}
We will be following the notations and settings of Section~\ref{section2-Prelim}, in particular those 
of Subsection~\ref{subsect:OrbiCherns}.
Our aim is to prove that the natural map 
$$
\psi_{\bullet} : \A^{\bullet}(\widehat X^\circ_{\mathcal C})^G \otimes \Q  \to  \A_{n-\bullet} (X^\circ) \otimes \Q
$$
can be used to equip $\A(X^\circ)$ with a ring structure. This can be achieved, as in~\cite{MR717614}
in the case of $\mathcal Q$-varieties, by using the following lemma.

\begin{lemma}
Assuming that $\mathcal C = \{ (U^\circ_{\alpha}, f_{\alpha} , X_{\alpha}^\circ) \}$ verifies Assumption~\ref{assume}, the map $\psi_{\bullet}$ is a group isomorphism.
\end{lemma}

\begin{proof}
The proof follows from the arguments of~\cite[Thm. 3.1]{MR717614}. 
Let $Z^\circ\in \A_{n-k}(X^\circ)$. We first show that $Z^\circ$ naturally gives rise
to an orbifold sheaf. 

Consider $(p_{\alpha}')^{-1}(Z)$ as a reduced subscheme of $V_{\alpha}$.

\begin{claim}\label{claim:ring}
The collection of $G_{\alpha}$-sheaves 
$$
\big\{ p_{\alpha}^*\big( \O_{(p_{\alpha}')^{-1}(Z^\circ)} \big)  \big\}
$$
is an orbifold sheaf.
\end{claim}

\noindent
\emph{Proof of Claim~\ref{claim:ring}.} It suffices to establish the compatibility 
condition for overlaps. To this end, consider another local orbifold chart

$$
\xymatrix{
X^\circ_{\beta}  \ar[rr]^{p_{\beta}}_{\text{adapted}}  && W^\circ_{\beta} \ar[rr]^{p'_{\beta}}_{\text{quasi-\'etale}}
&& X^\circ,
}
$$
where $W_{\beta}^\circ$ is smooth, with the resulting commutative diagram 
of fibre products:

$$
  \xymatrix{
   X^\circ_{\alpha\beta} \ar[rrrrrr]^{f_{\beta\alpha}} \ar[dd]_{f_{\alpha\beta}} \ar[drrrr]^{g_{\alpha\beta}} 
  &&&&&&  X^\circ_{\beta} \ar[d]^{p_{\beta}} \\
  &&&& W^\circ_{\alpha\beta} \ar[rr]^{p'_{\beta\alpha}} \ar[d]^{p'_{\alpha\beta}}  && W^\circ_{\beta} \ar[d]^{p'_{\beta}}  \\
  X^\circ_{\alpha} \ar[rrrr]^{p_{\alpha}}  &&&& W^\circ_{\alpha} \ar[rr]^{p'_{\alpha}} && X^\circ
.}
$$
Here, $W^\circ_{\alpha\beta} = W^\circ_{\alpha}\times_{X^\circ} W_{\beta}^\circ$ and $X^\circ_{\alpha\beta}$
is the normalization of $X^\circ_{\alpha}\times_{X^\circ} X^\circ_{\beta}$ with $g_{\alpha\beta}$
being the morphism induced by the universal property of fibre products. 

Now, since both $W_{\alpha}^\circ$ and $W^\circ_{\beta}$ are smooth and $p'_{\beta\alpha}$
and $p'_{\alpha\beta}$ are quasi-\'etale, it follows from the purity of branch locus that 
$p'_{\alpha\beta}$ and $p'_{\beta\alpha}$ are in fact \'etale. 
As a result we have
$$
(p'_{\alpha\beta})^*\big(  \O_{(p'_{\alpha})^{-1}(Z^\circ)}\big)     \cong    (p'_{\beta\alpha})^* 
\big( \O_{(p_{\beta}')^{-1}(Z^\circ)}  \big).
$$

On the other hand, from the commutativity of the diagram we find:

\begin{equation}\label{eq:firstEQ}
f^*_{\alpha\beta} \big(  p_{\alpha}^* \O_{(p_{\alpha}')^{-1}(Z^\circ)} \big) =  g_{\alpha\beta}^*\Big( 
  (p_{\alpha\beta}')^* \O_{(p_{\alpha}')^{-1}(Z^\circ)} \Big),
\end{equation}

\begin{equation}\label{eq:secondEQ}
f_{\beta\alpha}^*\big(  p_{\beta}^* \O_{(p_{\beta}')^{-1}(Z^\circ)}  \big)  =   g_{\alpha\beta}^*\Big( (p_{\beta\alpha}')^*
\O_{(p_{\beta}')^{-1}(Z^\circ)} \Big).
\end{equation}

\noindent
As the right hand side of \eqref{eq:firstEQ} and \eqref{eq:secondEQ} are isomorphic, so are the sheaves 
on the left hand side, as required. Compatibility over triple overlaps follow similarly. 
This proves the claim. \qed

\smallskip

The rest of the proof is now identical to the arguments of Mumford, cf.~\cite[p.~287]{MR717614}.
\end{proof}


 
\bibliographystyle{smfalpha}
\bibliography{biblio}

\end{document}